\documentclass[a4paper, 10pt,reqno,titlepage]{amsart}

\usepackage[foot]{amsaddr}
\usepackage{tikz}
\usepackage{tkz-euclide}
\usetkzobj{all}
\usetikzlibrary{arrows.meta}
\usetikzlibrary{quotes}
\usepackage{fullpage}
\usepackage{subcaption}
\usepackage[utf8]{inputenc}
\usepackage{amsmath}
\usepackage{amssymb}
\usepackage{amsfonts}
\usepackage{paralist}
\usepackage{braket}
\usepackage{placeins}
\usepackage{mathtools}
\usepackage{wrapfig}
\usepackage{caption}
\usepackage{hyperref}
\usepackage{amsthm}
\usepackage{xcolor}

\theoremstyle{plain} 
\newtheorem{thm}{Theorem}[section]
\newtheorem{lem}[thm]{Lemma}

\newtheorem{prop}[thm]{Proposition}

\theoremstyle{definition}
\newtheorem{defn}{Definition}[section]

\theoremstyle{remark}
\newtheorem{rem}[thm]{Remark}

\newcommand{\C}{\mathbb{C}}
\newcommand{\R}{\mathbb{R}}
\newcommand{\T}{\mathbb{T}}

\newcommand{\Z}{\mathbb{Z}}

\newcommand\bbA{{\mathbb{A}}}
\newcommand\bbB{{\mathbb{B}}}

\newcommand{\sfA}{{\mathsf{A}}}

\newcommand{\dr}{\mathrm{d}}

\newcommand{\de}{\partial}

\newcommand\restr[2]{{
  \left.\kern-\nulldelimiterspace 
  #1 
  \vphantom{\big|} 
  \right|_{#2} 
  }}

\def\g{\mathfrak{g}}

\def\sfA{\mathsf{A}}
\def\sfB{\mathsf{B}}
\def\sfC{\mathsf{C}}
\def\sfa{\mathsf{a}}
\def\sfb{\mathsf{b}}

\def\calY{\mathcal{Y}}
\def\calH{\mathcal{H}}
\def\calL{\mathcal{L}}
\def\calV{\mathcal{V}}
\def\calP{\mathcal{P}}
\def\calS{\mathcal{S}}
\def\coho{H^{\bullet}}

\def\calB{\mathcal{B}}
\def\calF{\mathcal{F}}
\def\calQ{\mathcal{Q}}

\begin{document}
\title{Theta Invariants of lens spaces via the BV-BFV formalism}
\author{Alberto S. Cattaneo${}^1$}
\author{Pavel Mnev${}^{2}$}
\author{Konstantin Wernli${}^1$}
\address{${}^1$Institut f\"ur Mathematik, Universit\"at Z\"urich, Winterthurerstrasse 190, 8053 Z\"urich, Switzerland}
\email{cattaneo@math.uzh.ch, konstantin.wernli@math.uzh.ch}
\address{${}^2$ University of Notre Dame, Notre Dame, IN 46556, USA}
\email{pmnev@nd.edu}

\thanks{A. S. C. and K. W. acknowledge partial support of SNF Grant No. 200020-172498/1. This research was (partly) supported by the NCCR SwissMAP, funded by the Swiss National Science Foundation, and by the COST Action MP1405 QSPACE, supported by COST (European Cooperation in Science and Technology). P. M. acknowledges partial support of RFBR Grant No. 17-01-00283a.  K. W. acknowledges partial support by the Forschungskredit of the University of Zurich, grant no. FK-16-093.}

\tikzset{middlearrow/.style={
        decoration={markings,
            mark= at position 0.5 with {\arrow{#1}} ,
        },
        postaction={decorate}
    }
}
\begin{abstract}
The goal of this paper is to investigate the Theta invariant --- an invariant of framed 3-manifolds associated with the lowest order contribution to the Chern-Simons partition function --- in the context of the quantum BV-BFV formalism. Namely, we compute the state on the solid torus to low degree in $\hbar$, and apply the gluing procedure to compute the Theta invariant of lens spaces. We use a distributional propagator which does not extend to a compactified configuration space, so to compute loop diagrams we have to define a regularization of the product of the distributional propagators, which is done in an \emph{ad hoc} fashion. Also, a polarization has to be chosen for the quantization process. Our results agree with results in the literature for one type of polarization, but for another type of polarization there are extra terms. 
\end{abstract}
\maketitle
\tableofcontents
\section{Introduction}
The goal of this paper is to investigate the Theta invariant - the first example of a perturbative Chern-Simons invariant - through the BV-BFV formalism\footnote{The letters B, F and V here stand for Batalin, Fradkin and Vilkovisky, who introduced what is now known as BV (\cite{Batalin1977, Batalin1981, Batalin1983}) and BFV \cite{Fradkin1975, Fradkin1977, Batalin1983a, Batalin1986} formalisms related to the quantization of gauge theories.}, which allows for perturbative quantization compatible with cutting and gluing (\cite{Cattaneo2017}). We will briefly discuss the BV-BFV formalism in Section \ref{sec:BVBFV}.\\  
The concept of perturbative Chern-Simons invariants goes back to the idea of Schwarz \cite{Schwarz1978,Schwarz2007} that physical quantities associated to topological quantum field theories should be topological invariants of the spacetime $M$ on which the theory is defined. Among physical quantities of interest for topological quantum field theories is the partition function. If the space of fields of the theory is $F_M$, and its action functional is $S_M \colon F_M \to \R$, the first attempt at defining the partition function is
\begin{equation}
Z_M = \int_{F_M}e^{\frac{i}{\hbar}S_M[\Phi]}. \label{eq:def_partition}
\end{equation}
On the right-hand side one is supposed to integrate over the space of fields $F_M$. However, apart from the case $\dim M=1$, there is hardly a case of interest where it has been possible to define a measure on $F_M$. Therefore, an alternative definition of $Z_M$ is required. There are different ideas for this in the literature, one of them being \emph{perturbative quantization}. Here one extrapolates the asymptotic behaviour in the $\hbar \to 0$ limit of an integral of the form $Z[\hbar] = \int_F \exp(i/\hbar S)$, where $F$ is a finite-dimensional manifold, to the infinite-dimensional case. In the finite-dimensional case, the result is a power series describing the asymptotics whose coefficients depend on the local behaviour of $S$ around critical points, and there is hope to define a similar power series in the infinite-dimensional case. Now, if the theory is topological, we expect the coefficients of this power series to be topological invariants of the spacetime. In Chern-Simons theory, this is not quite the case: It has been shown repeatedly (\cite{Axelrod1991, Axelrod1994, Kontsevich1994, Kuperberg1999, Bott1998, Lescop2004}) that the coefficients, starting with the lowest order, depend on choices that are made in the quantization process. However, it is possible to cancel the dependence on the choices, at the price of introducing a framing (i.e. a trivialization of the tangent bundle of the spacetime manifold). The result is an invariant of framed 3-manifolds known as Theta invariant. We review perturbative Chern-Simons quantization in more detail in Section \ref{sec:CS}. \\ 
\subsection{Main results}
In this paper we compute the contribution of the theta graph on lens spaces by the cutting and gluing method provided by the  quantum BV-BFV formalism, and one of the main purposes of this paper is to investigate the power of that method. However, the results that we obtain are interesting on their own. The main result the of the paper is the following. Let $(\g,\langle\cdot,\cdot\rangle)$ be a quadratic Lie algebra which allows a splitting into maximal isotropic subspaces $\g = V \oplus W$. Use the polarizations $\mathcal{P}, \calP'$ induced by this splitting (considering either $V$ or $W$ to be the fibers of the polarization) to perform BV-BFV quantization of the Chern-Simons theory with coefficient Lie algebra $\g$. Then we have the following result:
\begin{thm}
Consider the lens space $L_{p,q}^\varphi$ obtained from gluing $L_{p,q} = (S^1 \times D) \cup_{\varphi} (S^1 \times D) $, where $\varphi\colon S^1 \times S^1 \to S^1\times S^1$ is given by $\varphi = \begin{pmatrix}
m & p \\
n & q
\end{pmatrix}$. Then, the two-loop contribution to the state of split Chern-Simons theory on the lens space, obtained from gluing the states on two solid tori in polarizations $\mathcal{P}$ and $\calP'$  is given by 
 \begin{equation}
 w^{MT}_2 = e(\g)\left(\frac{1}{2}s(q,p) + \frac{q + m}{12p}\right)
 \end{equation}
 if  $V,W$ are subalgebras of $\g$\footnote{Hence the superscript $MT$ for Manin triple.} Here $s(q,p)$ denotes a Dedekind sum (see \eqref{eq:defDedekindSum}) and $e(\g)$ a Lie-algebra theoretic coefficient defined in \eqref{eq:defecc}.
 Otherwise, the two-loop contribution is 
 \begin{equation}
 w_2^{NMT} = w_2^{MT} + e'(\g)\left(\frac{1}{2}s(q,p) + \sum_{k=0}^{p-1}\eta_{S^1}(k/p)f(qk/p) + \eta_{S^1}(k/p)f(mk/p) + 
\frac{q + m}{2\pi^2}H_{1/p}\right)
 \end{equation}
 where  $e'(\g)$ is a Lie-algebra theoretic coefficients defined in \eqref{eq:defecc2}  $\eta_{S^1}$ and $f$ are functions on $S^1$ defined in \eqref{eq:circleprop} and \eqref{eq:def_f} respectively and $H_x$ denotes analytic extension of the harmonic numbers to $\R$. 
\end{thm}
We would like to emphasize that the result depends both on the choice of gluing diffeomorphism and splitting of $\g$. The dependence on the gluing diffeomorphism is to be expected: A change in the gluing diffeomorphism induces a change of framing of the solid torus, and hence the glued lens space. $w_2^{MT}$ changes in the expected way under such a change of framing, namely by $\frac{1}{12}$, if the framing is changed by one unit.  However, the dependence on the polarization is unexpected and has not been previously observed in the literature. These issues will be discussed in more detail in Sections \ref{sec:Comparing} and \ref{sec:Conclusion}.

\subsection{Plan of the paper}
 Sections \ref{sec:CS} and \ref{sec:BVBFV} contain a brief review of Chern-Simons theory and the BV-BFV formalism. 
 Since lens spaces admit a Heegaard splitting of genus 1, it is enough to compute the state - to the necessary order - on a solid torus. This is done in Section \ref{sec:SolidTorus}. In order to treat Chern-Simons as a perturbation of abelian BF theory, we make the assumption that the Lie algebra of coefficients admits a splitting into Lagrangian subspaces with respect to the invariant form used in the definition of the Chern-Simons form, see \ref{sec:SplitCS}. This determines a polarization on the space of boundary fields. As a gauge fixing we choose the axial gauge propagator. We then compute the effective action is Section \ref{sec:eff_action_torus}. Since the axial gauge propagator does not extend to a compactified configuration space, a different regularization has to be chosen for loop diagrams. In this note we restrict ourselves to a naive ad hoc regularization discussed in \ref{loops}. This issue is addressed in more detail in \cite{Wernli2018}. \\ 
In Section \ref{sec:Gluing} we explain how to compute the pairing of the states on two solid tori. In Section \ref{sec:EffectiveAction} we present the result for the effective action.  Remarkably, the result - both for the state on the solid torus and the weight of the Theta graph - depends on whether the subspaces are subalgebras or not. This seems similar to other results in the literature (\cite{Weinstein1987, Hawkins2008}) where such dependence on the polarization is discussed. If both subspaces are subalgebras, the result can be compared with the results by Kuperberg-Thurston-Lescop (see Section \ref{sec:Comparing}). Finally, in Section \ref{sec:Conclusion}, we present some conclusions and discuss further questions and related work. \\ 

\section{Perturbative Quantization of Chern-Simons theory}
\label{sec:CS}
We briefly review the idea behind perturbative quantization of Chern-Simons theory. 
\subsection{Chern-Simons Theory}
Chern-Simons is an example of a topological field theory that after the seminal paper by Witten (\cite{Witten1989}) received widespread attention throughout the mathematical physics and topology communities. It is a field theory defined on 3-dimensional manifolds with the space of fields $F_M$ the space of connections on the trivial principal $G$-bundle over $M$: 
\begin{equation}
F_M = Conn(M \times G) \cong \Omega^1(M,\g).
\end{equation}
Here $G$ is a Lie group and $\g$ its Lie algebra, subject to assumptions depending on the context. We identify the space of connections with $\g$-valued 1-forms on $M$. The action functional is \begin{equation}
S_M[A] = \int_M \frac{1}{2}\langle A, \dr A \rangle + \frac16 \langle A, [A, A] \rangle. \label{eq:CSaction}
\end{equation}
Here $\langle \cdot , \cdot \rangle$ is an invariant symmetric bilinear form on $\g$. The action \eqref{eq:CSaction} is the integral of the Chern-Simons form defined in \cite{Chern1974}, whence the name. 
An important feature of Chern-Simons theory is its gauge invariance under infinitesimal gauge transformations\footnote{Only the exponential of the Chern-Simons action, with the correct normalization, is invariant under large gauge transformations, but this is irrelevant for the perturbative treatment in this paper.} $A \to \dr_Ac$, where $c \in C^{\infty}(M,\g)$ (if $\de M = \emptyset$ or appropriate boundary conditions are chosen). Here $\dr_A c = \dr c + [A,c]$.  One can extend the theory to non-trivial bundles, but this requires more work. A thorough discussion of classical Chern-Simons theory can be found in the review papers by Freed,  \cite{Freed1995} for the case of trivializable bundles and \cite{Freed2002} for the nontrivializable case. 

\subsection{Perturbative quantization of gauge theories}
We remain very brief in this subsection. The interested reader is invited to study the references \cite{Reshetikhin2010, Mnev2017} for a thorough introduction to the perturbative quantization of gauge theories from the mathematical viewpoint, or the short introductory paper \cite{Polyak2005} (which is devoted exclusively to Chern-Simons theory).   \\ 

The quantity of interest in this paper is the Chern-Simons partition function $Z_M$. Its naive definition would be 
\[ Z_M = \int_{F_M}e^{\frac{i}{\hbar}S_M}, \]
but attempts at defining an appropriate measure on $F_M$ have failed (see \cite{Glimm1987} for a discussion of measures on field spaces in quantum field theory). In perturbative quantization, one tries to define $Z_M$ by extrapolating the behavior of a finite-dimensional oscillatory integral $Z = \int_F e^{\frac{i}{\hbar}S}$ in the $\hbar \to 0$ limit. It is well known that if $S$ has non-degenerate critical points, then the integral concentrates in a neighbourhood of them, and one can derive a series in powers of $\hbar$ describing the asymptotic behaviour. 
One can mimick the definition of this power series in the infinite-dimensional case if the critical points of $S_M$ are non-degenerate. However, the critical points of functionals invariant under a symmetry, such as the Chern-Simons functional, are never non-degenerate, so one needs an additional method from physics, called \emph{gauge fixing}. There are different variants of this method, the most commonly used being the Faddeev-Popov (FP) ghosts (\cite{Faddeev1967}) and the BRST formalism\footnote{Sometimes also BRS formalism, it was introduced independently by Becchi, Rouet and Stora (\cite{Becchi1975, Becchi1976}) and Tyutin (\cite{Tyutin1976}).}. The idea is to embed the space of fields $F_M$ in the degree 0 part of a graded vector space\footnote{In other examples the space of fields can be non-linear, but we will only deal with the linear case, which is also easier in view of quantization.} $\calF_M$, and define a new functional $\calS_M \colon \calF_M \to \R$ such that $\calS_M$ has non-degenerate critical points, and $\calS_M|_{F_M} = S_M$. \\ 
Both the FP ghosts and the BRST formalism are subsumed in the BV formalism named after Batalin and Vilkovisky, who introduced it in \cite{Batalin1977, Batalin1981, Batalin1983}. We will briefly discuss the BV formalism and its adaptation to the case with boundary, the BV-BFV formalism, in the next section.
\subsection{Perturbative Chern-Simons theory and invariants of 3-manifolds and links}
There is a considerable number of invariants of 3-manifolds or rational homology spheres or possibly knots and links in 3-manifolds that are related or supposedly related to perturbative Chern-Simons theory, including the Kontsevich integral \cite{Froehlich1989, Kontsevich1993, Bar-Natan1995}, the Le-Murakami-Ohtsuki invariant \cite{Le1998}, the Aarhus integral \cite{Bar-Natan2002}, the Kontsevich-Kuperberg-Thurston-Lescop  invariant of rational homology spheres\cite{Kontsevich1994, Kuperberg1999, Lescop2004, Lescop2004a}, the Axelrod-Singer invariants \cite{Axelrod1991, Axelrod1994} associated to acyclc flat connections, the Bott-Cattaneo \cite{Bott1998, Bott1999} invariants of rational homology spheres, the effective action approach by the first two authors (\cite{Cattaneo2008}) and of course the conjectural asymptotic expansion of the Reshetikhin-Turaev invariant \cite{Reshetikhin1991}. The authors hope that the methods of cutting and gluing diagrams introduced in \cite{Cattaneo2017}, and used in this paper to compute weights of theta graphs, can help to shed light on the various known and conjectured relations between these invariants in the future. 
\section{Review of BV-BFV quantization}
\label{sec:BVBFV}

The BV formalism is generally considered as the most powerful gauge-fixing formalism, as it can deal with a large class of gauge theories, especially also those where the gauge symmetries do not close off-shell. Detailed introductions can be found in \cite{Mnev2008, Mnev2017, Fiorenza2003}. In view of functorial quantization, an important fact is that it admits a natural extension to manifolds with boundary, the BV-BFV formalism, that is compatible with cutting and gluing both at the classical level and the level of perturbative quantization (\cite{Cattaneo2014, Cattaneo2017}). We briefly review the main concepts. For details and proofs we refer to the references above.
\subsection{BV formalism} 
In the BV formalism one embeds the space of states $F_M$ into an \emph{odd symplectic} $\Z$-graded vector space $(\calF_M,\omega)$, where the odd symplectic form $\omega$ is required to have $\Z$-degree $-1$. The $\Z$-degree is referred to as \emph{ghost number}. One needs to find a BV  action $\calS_M$ such that $\restr{\calS_M}{F_M}=S_M$, which satisifies the Classical Master Equation (CME) $(\calS,\calS)=0$, where $(\cdot,\cdot)$ is the odd Poisson bracket associated to $\omega$.  Notice that if $\calS_M$ has a hamiltonian vector field $\calQ_M$ then $\calQ_M^2 = 0$, from the CME. This leads to the following definition of BV theory (\cite{Cattaneo2014}): 
\begin{defn}
A \emph{BV vector space} is a quadruple $(\calF,\omega,\calQ,\calS)$ where $\calF$ is a $\Z$-graded vector space, $\omega$ is a symplectic form of degree $-1$, $\calQ$ is a degree $+1$ vector field, and $\calS$ is a function of degree 0, such that $\calQ^2 =0$ and 
\begin{equation}
\iota_\calQ\omega = \delta \calS. \label{eq:Hamiltonian}
\end{equation}
A \emph{$d$-dimensional (linear) BV theory} is an association of a $BV$ vector space $(\calF_M,\omega_M,\calQ_M,\calS_M)$ to every $d$-dimensional manifold $M$.
\end{defn}
Here the de Rham differential on $\calF$ is denoted by $\delta$. Equation \eqref{eq:Hamiltonian} says that $\calQ$ is the Hamiltonian vector field of $\calS$ and together with $\calQ^2 = 0$ this implies the CME. We will only be interested in local theories. Roughly this means that $\calF_M$ is given by sections of a sheaf over $M$, and the other objects are given by integrals of functions of jets of the fields. \\ 
In order to quantize one requires also the Quantum Master Equation (QME) \begin{equation}\label{eq:QME}
\Delta(e^{\frac{i}{\hbar}\calS}) = 0 \Leftrightarrow \frac{1}{2}(\calS,\calS) + i\hbar\Delta \calS = 0.
\end{equation}
Here $\Delta$ is the BV Laplacian associated to $\omega$, which in Darboux coordinates $(x^i,p_i)$ is given by $\sum_i \pm\frac{\de}{\de x^i}\frac{\de}{\de p_i}$\footnote{As always, the formalism is developed for finite-dimensional $F_M$, and then needs to adapted to the infinite-dimensional case.}\footnote{In this paper we completely ignore the (important) distinction between functions and half-densities, see e.g. \cite{Khudaverdian2004}.}.
To define the partition function $\psi_M$ one needs to choose a Lagrangian $\calL_M \subset \calF_M$, such that the partition function restricted to that Lagrangian has a unique critical point. This is the choice of gauge fixing. The partition function is then defined by the BV integral \begin{equation}\label{eq:BVintegral}
\psi_M = \int_{\calL_M}e^{\frac{i}{\hbar}\calS}.
\end{equation}
Under the assumption that the QME is satisfied, the partition function $\psi_M$ defined by \eqref{eq:BVintegral}  is invariant under deformations of the gauge-fixing Lagrangian $\calL_M$. \\ 
In the infinite-dimensional case usually at hand in quantum field theory, the BV Laplacian is ill-defined, and has to be regularized, whereas the BV integral \eqref{eq:BVintegral} has to be replaced by its formal asymptotic version, which yields a sum over Feynman graphs. 
\subsubsection{Effective Action and Residual fields}
In certain cases, it is not possible to find directly a Lagrangian which satisfies the requirement that the action restricted to it  has a unique critical point. A good example is the case of abelian BF theory. In $d$ dimensions, the BV space of fields is $\calF_M = \Omega^\bullet(M)[1] \oplus \Omega^\bullet(M)[d-2]$, and the BV action functional is 
\begin{equation}
\calS_M[\sfA,\sfB] = \int_M \sfB \wedge \dr\sfA.
\end{equation}
The critical points are given by closed forms, and the gauge symmetries are given by shifting $\sfA,\sfB$ by exact forms. Hence if the de Rham cohomology $H^\bullet(M)$ is nontrivial, there are critical points which are inequivalent under gauge transformations, also known as zero modes. \\
The solution to this problem is to choose a BV space of residual fields $\calV_M$ and a splitting $F_M  = \calV_M \times \calY$, such that one can find a gauge-fixing Lagrangian $\calL_M \subset \calY$. Elements of $\calV_M$ are known as residual fields, zero modes, infrared fields, or slow fields, whereas the elements of $\calY$ are called fluctuations, fast fields, or ultraviolet fields. The partition function $\psi_M$ gets replaced by an \emph{effective action} $\psi_M(\mathsf{x})$, which is a function of the residual fields and formally defined via BV pushforward: 
\begin{equation}
\psi_M(\mathsf{x}) = \int_{\xi \in \calL_M \subset \calY}e^{\frac{i}{\hbar}\calS[\mathsf{x},\xi]}.\label{eq:state_residualfields}
\end{equation}
In the case of abelian BF theory one can choose as residual fields representatives of the cohomology: $\calV_M = H^\bullet(M)[1] \oplus H^\bullet(M)[d-2]$. One way to do this is to pick a Riemannian metric and use the harmonic representatives, a possible choice of gauge-fixing Lagrangian is then given by $\dr^*$-exact forms. In the finite-dimensional case, the QME for the BV action implies that the effective action is $\Delta_{\calV_M}$-closed, i.e. closed with respect to the BV Laplacian on residual fields, and changes by a $\Delta$-exact term under a deformation of the gauge fixing Lagrangian. 
\subsection{BV-BFV extension on manifolds with boundary} 
A crucial feature of the BV formalism is that it admits an extension to manifolds that is compatible with cutting and gluing, in the form of the BV-BFV formalism. The classical case is discussed in \cite{Cattaneo2014}. An introduction can also be found in \cite{Schiavina2015}.  The perturbative quantization was introduced in \cite{Cattaneo2017}. Further discussions of the gauge theories on manifolds with boundary both at classical and quantum level can be found in \cite{Cattaneo2011,Cattaneo2015a}. For a brief review see also \cite{Cattaneo2016, Cattaneo2018a}.
 We briefly review the main notions in this subsection.  
\subsubsection{Classical case}
First one needs the definition of BFV vector space (for background on the BFV complex see \cite{Stasheff1997, Schaetz2010}). 
\begin{defn}
A BFV vector space is a triple $(\calF^\de,\omega^\de,\calQ^\de)$, where $\calF^\de$ is a $\Z$-graded vector space, $\omega^\de$ is a symplectic form of degree $0$ on $\calF^\de$, and $\calQ^\de$ is a degree $+1$ vector field which is symplectic and satisfies $(\calQ^\de)^2=0$. 
\end{defn} 
By degree reasons $\calQ^\de$ is automatically Hamiltonian with Hamiltonian function $\calS^\de$.\\ 
Now the idea is to associate to the boundary of a manifold a BFV vector space, and to the bulk a suitable generalization of a BV vector space which is compatible with the BFV data associated to the boundary. The solution is the notion of BV-BFV vector space as introduced in \cite{Cattaneo2014}. 
\begin{defn}
Let $\calF^\de = (\calF^\de,\omega^\de= \delta \alpha^\de,\calQ^\de)$ be a BFV vector space with exact symplectic form $\omega^\de$. A BV-BFV vector space over $\calF^\de$ is a quintuple 
$(\calF, \omega,\calQ,\calS,\pi)$, where $\calF$ is a $\Z$-graded vector space, $\omega$ is a degree $-1$ symplectic form, $\calQ$ is a degree $+1$ vector field, $\calS$ is a degree 0 function on $\calF$ and $\pi \colon \calF \to \calF^\de$ is a surjective submersion such that $\calQ^2 = 0$, $\delta\pi \calQ = \calQ^\de$ and 
\begin{equation}
\iota_Q\omega = \delta S + \pi^* \alpha^\de. \label{eq:mCME}
\end{equation}
\end{defn}
\begin{rem}
Equation \eqref{eq:mCME} should be thought of as a generalization of the CME. In fact, in the case $\calF^\de = \{0\}$ the definition of BV-BFV vector space reduces to that of an ordinary BV vector space. 
\end{rem}

We are now ready to define the notion of classical BV-BFV theory.  Namely, a $d$-dimensional BV-BFV theory associates to closed $(d-1)$-dimensional manifold $\Sigma$ an exact BFV vector space $\calF^\de_\Sigma$ and to a $d$-dimensional manifold $M$ with boundary $\de M$ a BV-BFV vector space $\calF_M$ over the BFV vector space $\calF^\de_{\de M}$. Some remarks are in order: 
\begin{rem}
\begin{enumerate}[i)]
\item The exactness condition (like the linearity condition) can be relaxed. See \cite[Remark 3.3]{Cattaneo2014} for a short discussion. 
\item We also require locality from BV-BFV theories. That means the vector spaces $\calF, \calF^\de$ are given by sections of a sheaf and as such are typically infinite-dimensional (over $\R$ or $\C$). They can be equipped with natural Banach or Fr\'echet topologies depending on the situation. 
\item With enough care, a BV-BFV theory yields a functor from a cobordism category (perhaps equipped with extra structure) to a category of vector spaces where composition is given by (homotopy) fibered product. Again, we refer to \cite[Section 4]{Cattaneo2014} for a more detailed discussion. 
\end{enumerate}
\end{rem}
\subsubsection{Quantization}
Classical BV-BFV theories admit a perturbative quantization that is compatible with cutting and gluing (\cite{Cattaneo2017}). Roughly, the idea is to combine geometric quantization of $(\calF^\de,\omega)$ and perturbative quantization as in Equation \eqref{eq:state_residualfields}. We briefly outline the main steps. 
\begin{enumerate}[i)]
\item Choose a polarization $\calP$ on $(\calF^\de,\omega)$ with smooth leaf space $\calB^\calP$. In our case it will be enough to find a splitting  $\calF^\de = \calB_1^\calP \times \calB_2^\calP$, where both $\calB_1^\calP$ and $\calB_2^\calP$ are Lagrangian subspaces of $\calF^\de$. One can choose either $\calB_i^\calP$ as base space or fibers of the polarization respectively.
\item If necessary, change $\alpha^\de$ by an exact term such that it vanishes on the fibers of $\calP$. To preserve Equation \eqref{eq:mCME} one has to change $\calS$ by a boundary term. 
\item Chooose a section $\sigma$ of $\calF \to \calF^\de \to \calB^\calP$, and a splitting $\calF = \sigma(\calB^\calP) \times \calY$ (subject to certain requirements discussed in \cite{Cattaneo2017}). We immediately suppress $\sigma$ from the notation. 
\item Then, proceed to split $\calY = \calV \times \calY'$ such that there is a Lagrangian $\calL \subset \calY'$ on which the action $\calS$ has a unique critical point. We now have a splitting $\calF = \calB^\calP \times \calV \times \calY'$, and accordingly we write $\mathsf{X} = \mathbb{X}+\mathsf{x} + \xi$ for $\mathsf{X}\in \calF$.
\item Finally, define the state or partition function formally by 
\begin{equation}
\psi( \mathbb{X},\mathsf{x}) = \int_{\xi \in \calL_M \subset \calY}e^{\frac{i}{\hbar}\calS[\mathbb{X},\mathsf{x}, \xi]}.\label{eq:state_residualandboundaryfields}
\end{equation}
\end{enumerate}
A short discussion is in order. \\
In the finite-dimensional case, Equation \eqref{eq:state_residualandboundaryfields} is an example of a ``BV pushforward in families'' introduced in \cite{Cattaneo2017}. In the infinite-dimensional case, definition \eqref{eq:state_residualandboundaryfields} has to be interpreted via the Feynman graphs and rules discussed in \cite{Cattaneo2017} and \cite{Cattaneo2017a}. We will briefly review the ones relevant for this paper below. \\
The state \eqref{eq:state_residualandboundaryfields} is a  functional on both $\calV$ and $\calB^\calP$. We think of it as an element of $\widehat{\calH}^\calP = \widehat{S}\calV^* \otimes \calH^\calP$, where $\calH^\calP$ is a certain space of functionals on $\calB^\calP$ that should be thought of as a geometric quantization of $(\calF^\de,\omega)$, and $\widehat{S}$ denotes the formal completion of the symmetric algebra. Also the boundary action $S^\de$ can be quantized and yields a coboundary operator $\Omega^\calP$ on $\calH^\calP$. The precise construction of this space and the coboundary operator $\Omega$ are not relevant for this paper, the interested reader is again referred to  \cite[Section 4.1]{Cattaneo2017}. Also $\calV$ carries a coboundary operator, the BV Laplacian $\Delta_\calV$, and the state is a cocycle in the bicomplex $\calH^\calP$: 
\begin{equation}
(\hbar^2\Delta_\calV + \Omega^\calP)\psi = 0. \label{eq:mQME}
\end{equation}
Equation \eqref{eq:mQME} is called the modified Quantum Master Equation (mQME). \\ 
An important feature of the quantum BV-BFV formalism is that the perturbative expansions associated to manifolds with boundary can be glued together using a form of the BKS\footnote{For Blattner, Kostant, Sternberg. See \cite{Bates1997}} pairing discussed in \cite{Cattaneo2017}. The gluing procedure is important to this paper and we will discuss it in more detail below in Section \ref{sec:Glueing}. 
\section{Split Chern-Simons Theory on the solid torus}
We now apply the BV-BFV formalism to the case of split Chern-Simons theory, in the example of the solid torus. Partly this example was already discussed in \cite{Cattaneo2017a}. In this paper we use the axial gauge fixing to explicitly compute the Feynman diagrams in low orders on the solid torus and, via gluing, also on lens spaces.
\subsection{Split Chern-Simons Theory}\label{sec:SplitCS}
Split Chern-Simons theory as an example of a perturbation of abelian BF theory was first considered in \cite{Cattaneo2017} and investigated more thoroughly in \cite{Cattaneo2017a}, see also \cite{Wernli2018}. Let $\g$ be a Lie algebra with  symmetric ad-invariant bilinear form $\langle\cdot , \cdot\rangle$. Consider the BV-extended Chern-Simons action functional 
$$ \calS[\sfC] = \int_M \frac{1}{2}\left\langle \sfC,\dr\sfC \right\rangle + \frac{1}{6}\left\langle \sfC,\left[\sfC,\sfC\right]\right\rangle = \int_M\frac{1}{2} B_{ab}\sfC^a\sfC^b + \frac{1}{6}f_{abc}\sfC^a\sfC^b\sfC^c $$
where $\sfC \in \calF_M = \Omega^{\bullet}(M,\g)[1]$ is the superfield, i.e. an inhomogeneous Lie algebra-valued differential form, and $\langle \cdot , \cdot \rangle$ and $[\cdot,\cdot ]$ denote extensions of the bilinear form and the Lie bracket to Lie algebra-valued differential forms. In the second expression we pick a basis $e_a$ of $\g$ and define $B_{ab} = \langle e_a, e_b \rangle$, $f_{abc} = \langle e_a, [e_b,e_c] \rangle$ and expand the field as $\sfC = \sfC^ae_a$. In \cite{Cattaneo2014} it is shown that this is an example of a BV-BFV theory.  \\
Now assume that $\g$ admits a splitting $\g = V \oplus W$ into maximal isotropic subspaces. We choose a basis $\xi_i$ of $V$ and a dual basis $\xi^i$ of $W$. Then the space of fields splits as 
$\Omega^\bullet(M,\g) = \Omega^\bullet(M,V) \oplus \Omega^\bullet(M,W)$ and the superfield $\sfC$ splits as $\sfA + \sfB = \sfA^i\xi_i + \sfB_i\xi^j$. Integrating by parts one can rewrite the action as 
\begin{equation}\calS[\sfA,\sfB] = \int_M \langle \sfB, \dr \sfA \rangle + \frac{1}{6}\langle \sfA, [\sfA,\sfA] \rangle + \frac{1}{2}\langle \sfB, [\sfA,\sfA] \rangle + \frac{1}{2} \langle \sfA, [\sfB,\sfB] \rangle + \frac{1}{6} \langle \sfB, [\sfB,\sfB]\rangle \label{eq:SplitCSAction}
\end{equation}
so it becomes a ``BF-like'' theory with interaction term 
\begin{align*}\calV(\sfA,\sfB) &= \frac{1}{6}\langle \sfA, [\sfA,\sfA] \rangle + \frac{1}{2}\langle \sfB, [\sfA,\sfA] \rangle + \frac{1}{2} \langle \sfA, [\sfB,\sfB] \rangle + \frac{1}{6} \langle \sfB, [\sfB,\sfB]\rangle \\
&= \frac{1}{6}g_{ijk}\sfA^i\sfA^j\sfA^k + \frac{1}{2}g^i_{jk}\sfB_i\sfA^j\sfA^k + \frac12 h_i^{jk}\sfA^i\sfB_j\sfB_k + \frac{1}{6}h^{ijk}\sfB_i\sfB_j\sfB_k
\end{align*}
where we introduced the structure constants $g_{ijk},g^i_{jk},h_i^{jk},h^{ijk}$ defined by applying $\langle\cdot,[\cdot,\cdot]\rangle$ to the basis vectors $\xi_i,\xi^i$. If  $(\g,V,W)$ form a a quasi-Manin triple (i.e. $V$ is a subalgebra), by isotropy the interaction term simplifies to
\begin{equation}
\calV(\sfA,\sfB) =  \frac{1}{2}\langle \sfB, [\sfA,\sfA] \rangle + \frac{1}{2} \langle \sfA, [\sfB,\sfB] \rangle + \frac{1}{6} \langle \sfB, [\sfB,\sfB]\rangle.\label{eq:SplitCSquasiManinTriple}
\end{equation}
In the special case where $(\g,V,W)$ is a Manin triple (i.e. $V,W$ are subalgebras), by isotropy we get that the interaction term simplifies to
\begin{equation}
\calV(\sfA,\sfB) =  \frac{1}{2}\langle \sfB, [\sfA,\sfA] \rangle + \frac{1}{2} \langle \sfA, [\sfB,\sfB] \rangle \label{eq:SplitCSManinTriple}
\end{equation}
\label{sec:SolidTorus}
We now specialize to the case of the solid torus $S^1 \times D$. We give it coordinates $x = (t,z) \in [0,1) \times \{z \in \C | |z| \leq 1\}$. 
\begin{rem}\label{framingsolidtorus}
Notice that by choosing global coordinates we also choose a trivialization of its tangent bundle - a framing. We will fix this framing on the solid torus. However, as we glue lens spaces from solid tori, the framing of the glued lens space will depend on the gluing diffeomorphism. See also Section \ref{sec:Gluing}.
\end{rem}

\subsection{Polarization}
In split Chern-Simons theory, the space of boundary fields splits as 
\begin{equation}
\calF^\de = \Omega^\bullet(\de M, V) \oplus \Omega^\bullet(\de M, W).
\end{equation} 
By the isotropy condition this is a splitting into Lagrangian subspaces, so we can use either of them as base or fibers of the polarization. The coordinate on the base is denoted by a blackboard bold letter $\bbA$ or $\bbB$, and we speak of $\bbA$- or $\bbB$-representation respectively\footnote{If one thinks of Chern-Simons theory as an AKSZ theory, this amounts to lifting a target polarization instead of a source polarization.}. This terminology comes from the $p$- and $q$-representations in Quantum Mechanics. \\
Since $\de M \cong S^1 \times S^1$ is connected, the only choice is between the $\bbA$- or the $\bbB$-representation on $\de M$. For computations we will use the $\bbA$-representation, i.e. we will split the $\sfA$-field as 
$$ \sfA = \widehat{\sfA} + \widetilde{\bbA}$$
where $\widetilde{\bbA}$ is an extension of $\bbA = \iota^*\sfA$ to the bulk. The computations for the $\bbB$-representation are exactly the same, one just needs some care in translating the results. See Section \ref{sec:Brep}.
\subsection{Residual fields and fluctuations}
The minimal space of residual fields for this polarization is $$\calV_M = H^{\bullet}(M,\de M)\otimes V \oplus H^{\bullet}(M) \otimes W \ni (\sfa,\sfb).$$ 
As representatives of the cohomology we choose 
\begin{align*}
\chi^1 &= 1 \\ 
\chi^2 &= dt \\
\chi_1 &= \mu dt \\ 
\chi_2 &= \mu \\
\end{align*}
where $\mu = \frac{1}{2\pi i}d\bar{z}dz$ is unit volume form on the disk,  $dt$ is the coordinate differential on $S^1$ and we omit wedge symbols. Notice that 
$$\int_M \chi^i\chi_j = \int_M \chi_j\chi^i = \delta^i_j$$
i.e. $\chi^i$ is the dual basis to $\chi_j$. We can then expand 
\begin{align*}
\sfa &= z^i\chi_i \\
\sfb &= z_i^+\chi^i\\
\end{align*}
where $z^i,z_i$ are linear $V$(resp. $W$)-valued functions on $H^{\bullet}(M,\de M)$ and $H^{\bullet}(M)$. 
The fluctuations $\alpha, \beta$ are then defined by 
\begin{align*}
\widehat{\sfA} &= \alpha +  \sfa \\
\widehat{\sfB} &= \beta + \sfb
\end{align*}
Notice that the fluctuations depend on the extension of $\bbA$, whereas the backgrounds $\sfa,\sfb$ do not. As discussed in \cite{Cattaneo2015a} one now chooses a discontinuous extension of $\bbA$ which drops to zero outside the boundary, which we also denote by $\bbA$, abusing notation. 
This leads to a splitting of the action as 
\begin{equation}
\calS = \calS_0 + \calS^{back} + \calS^{source} + \calS^{int}
\end{equation}
where 
\begin{align*}
\calS_0 &= \int_M \widehat{\sfB}\widehat{\dr\sfA} \\
\calS^{back} &= \int_{\de M } \sfb\bbA \\ 
\calS^{source} &= \int_{\de M} \beta \bbA \\
 \calS^{int} &= \calV(\sfA,\sfB). 
\end{align*}

\subsection{Axial gauge propagator}\label{sec:AxialGauge}
Axial gauge propagators on product manifolds were discussed in \cite{Cattaneo2015a} (Appendix C.1), but the idea is older and goes back to works of Fr\"ohlich \cite{Froehlich1989}. \\ 
The identity these distributional propagators satisfy is  
\begin{equation}
\dr\eta_M = \delta_M^{(d)}(x_1,x_2) + (-1)^{d-1} \sum_i (-1)^{d\cdot\deg \chi_i}\pi_1^*\chi_i\pi_2^*\chi^i \label{distrprop}
\end{equation}
On the solid torus,  we have two choices for a distributional propagator.  The horizontal propagator is 
\begin{equation}
\label{etahor}
\eta^{hor}((z_1,t_1),(z_2,t_2)) = \eta_D(z_1,z_2)\delta^{(1)}_{S^1}(t_1,t_2) + \mu_1\eta_{S^1}(t_1,t_2)
\end{equation}
while the axial propagator is 
\begin{equation}
\label{etaax}
\eta^{ax}((z_1,t_1),(z_2,t_2)) = \delta^{(2)}_D(z_1,z_2)\eta_{S^1}(t_1,t_2) + \eta_{D}(z_1,z_2)(dt_1-dt_2)
\end{equation}
where $\eta_D$ and $\eta_{S^1}$ are propagators on the the disk and the circle, respectively.
\begin{lem}
These  propagators satisfy
\begin{equation}
\dr\eta = \delta_M^{(3)}(x_1,x_2) + \sum_i (-1)^{\deg \chi_i}\pi_1^*\chi_i\pi_2^*\chi^i = \delta_M^{(3)}(x_1,x_2) - \mu_1dt_1 + \mu_1dt_2.
\end{equation}
\end{lem}
\begin{proof}
First note that in the distributional sense we have 
\begin{align*}
\dr_D\eta_D &= \delta^{(2)}_D(z_1,z_2) - \mu_1 \\
\dr_{S^1}\eta_{S^1} &= \delta^{(1)}_{S^1}(t_1,t_2) - (dt_1 - dt_2) \\
\dr_{S^1}\delta^{(1)}_{S^1} &= 0 
\end{align*}
where the first two identities are \eqref{distrprop} for $D$ and ${S^1}$ respectively, and the third is for dimensional reasons. 
Using this, we evaluate (omitting the arguments)
\begin{align*}
\dr\eta^{hor} &= \dr \eta_D \delta_{S^1} + \mu_1 \dr \eta_{S^1} \\
&= (\delta_D - \mu_1)\delta_{S^1} + \mu_1 (\delta_{S^1} - (dt_1 - dt_2) ) \\ 
&= \delta_D \delta_{S^1} - \mu_1(dt_1 - dt_2)  \\
&= \delta_M - \mu_1dt_1 + \mu_1dt_2
\end{align*}
and 
\begin{align*}
\dr\eta^{ax} &= \delta_D(\delta_{S^1} - dt_1 + dt_2) + (\delta_D-\mu_1)(dt_1 - dt_2) \\ 
&= \delta_M - \mu_1dt_1 + \mu_1dt_2. 
\end{align*}
\end{proof}
We will now state some other desirable properties the propagators have. 
\begin{prop}\label{prop:propprops}
Suppose $\int_{D_2} \eta_{D,{12}}\mu_2 = 0$, $\int_{D_2}\eta_{D,12}\eta_{D,23} = 0$ and $\int_{S^1_2}\eta_{S^1,12}dt_2=0$, $\int_{\de D_2}\eta_{D,12} =1$ Then the following identities hold for $\eta \in \{\eta^{ax},\eta^{hor}\} $: 
\begin{enumerate}
\item $\int_1 \eta_{12} = 0$, 
\item $\int_1 dt_1\eta_{12} = 0$, 
\item $\int_2 \eta_{12} dt_2 = 0$,
\item $\int_2 \eta_{12} \mu_2 = 0$,
\item $\int_2 \eta_{12} \mu_2dt_2 = 0$, 
\item $\int_2 \eta_{12} \eta_{23} = 0$, 
\item $\int_{2,\de} \eta_{12} = 1.$
\end{enumerate}
\end{prop}
\begin{proof}
All identities are direct applications of the assumptions together with some degree counting (notice a form needs to have bidegree (2,1) to contribute to the integral). 
\end{proof} 
\subsection{The state}
For reasons discussed below, from now on we only work with the horizontal propagator. We fix the disk and circle propagators 
\begin{align}
\eta_{S^1}(s,s') &= ((s-s')) = \begin{cases} 0 & s-s' \in \Z \\
s-s' - \lfloor s-s' \rfloor - 1/2 & s-s' \notin Z
\end{cases} \label{eq:circleprop}\\
\eta_D(z,w) &= \frac{1}{2\pi}d\mathrm{arg}(z-w) + \frac{1}{2\pi }d\mathrm{arg}(1-\bar{z}w) + \frac{1}{4\pi i }(\bar{z}dz - z d\bar{z})
\end{align} Thus we determine the BV gauge-fixing Lagrangian $\calL$ and therefore the state by the formal BV integral  
\begin{equation}
\widehat{\psi} = \int_{(\alpha,\beta)\in \calL} e^{\frac{i}{\hbar}\calS[\sfA,\sfB]}
\end{equation}
which is evaluated by Feynman graphs and rules as in \cite{Cattaneo2017,Cattaneo2017a}, which we briefly recall below. 
\subsubsection{Feynman graphs and rules} In this case the Feynman graphs and rules are as follows. Admissible graphs are directed connected graphs with a trivalent and univalent vertices. The trivalent vertices can have any number (between 0 and 3) of outgoing and incoming half-edges. Outgoing half-edges represent $\sfA$ fields while incoming half-edges represent $\sfB$ fields.  Univalent vertices can have only incoming half-edges and are decorated by a boundary field $\bbA$. See figure \ref{fig:vertices}. Incoming half-edges must either be connected to an outgoing half-edge or a $\sfb$ residual field (a half-edge ending in a residual field is thought of as a decoration of the vertex incident to the half-edge). Outgoing half-edges must either be connected to an incoming half-edge or an $\sfa$ residual field.
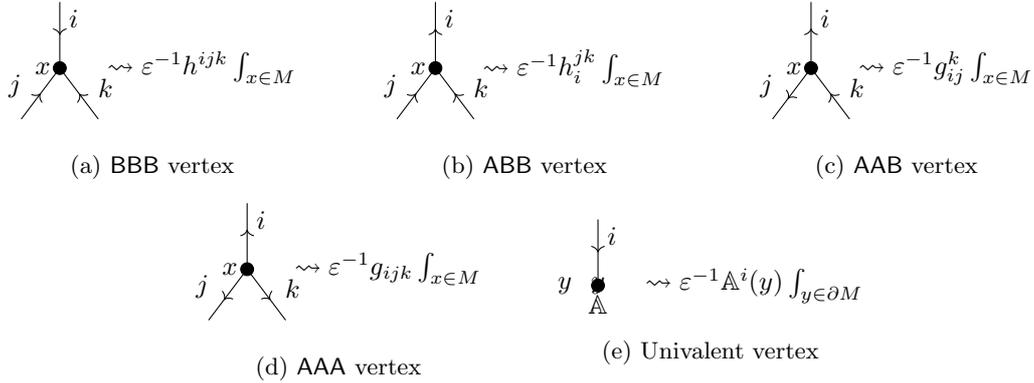
\begin{figure}[h]
\centering
\begin{subfigure}{0.3\textwidth}
\centering
\begin{tikzpicture}[scale=2]
\node[shape=coordinate,label=left:{$x$}] (O) at (0,0) {};
\node  at (0,0.5)  {}
edge[middlearrow={>}, "$i$" near start] (O);
\node  at (-0.3,-0.4) {}
edge[middlearrow={>},"$j$" near start] (O);
\node  at (0.3,-0.4)  {}
edge[middlearrow={>},"$k$"' near start] (O);
  \tkzDrawPoint[color=black,fill=black,size=12](O)
  \node[shape=coordinate, label=right:{$\leadsto\varepsilon^{-1} h^{ijk}\int_{x \in M}$}] at (0.25,0) {};
  \end{tikzpicture}
    \caption{$\sfB\sfB\sfB$ vertex}
  \end{subfigure}
  \begin{subfigure}{0.3\textwidth}
  \centering
  \begin{tikzpicture}[scale=2]
\node[shape=coordinate,label=left:{$x$}] (O) at (0,0) {};
\node  at (0,0.5)  {}
edge[middlearrow={<}, "$i$" near start] (O);
\node  at (-0.3,-0.4) {}
edge[middlearrow={>},"$j$" near start] (O);
\node  at (0.3,-0.4)  {}
edge[middlearrow={>},"$k$"' near start] (O);
  \tkzDrawPoint[color=black,fill=black,size=12](O)
  \node[shape=coordinate, label=right:{$\leadsto \varepsilon^{-1}h_i^{jk}\int_{x \in M}$}] at (0.25,0) {};
  \end{tikzpicture}
  \caption{$\sfA\sfB\sfB$ vertex}
  \end{subfigure}  
 \begin{subfigure}{0.3\textwidth}
 \centering
\begin{tikzpicture}[scale=2]
\node[shape=coordinate,label=left:{$x$}] (O) at (0,0) {};
\node  at (0,0.5)  {}
edge[middlearrow={<}, "$i$" near start] (O);
\node  at (-0.3,-0.4) {}
edge[middlearrow={<},"$j$" near start] (O);
\node  at (0.3,-0.4)  {}
edge[middlearrow={>},"$k$"' near start] (O);
  \tkzDrawPoint[color=black,fill=black,size=12](O)
  \node[shape=coordinate, label=right:{$\leadsto \varepsilon^{-1}g_{ij}^k\int_{x \in M}$}] at (0.25,0) {};
  \end{tikzpicture}
  \caption{$\sfA\sfA\sfB$ vertex}
\end{subfigure}
\begin{subfigure}{0.3\textwidth}
\centering
\begin{tikzpicture}[scale=2]
\node[shape=coordinate,label=left:{$x$}] (O) at (0,0) {};
\node  at (0,0.5)  {}
edge[middlearrow={<}, "$i$" near start] (O);
\node  at (-0.3,-0.4) {}
edge[middlearrow={<},"$j$" near start] (O);
\node  at (0.3,-0.4)  {}
edge[middlearrow={<},"$k$"' near start] (O);
  \tkzDrawPoint[color=black,fill=black,size=12](O)
  \node[shape=coordinate, label=right:{$\leadsto \varepsilon^{-1}g_{ijk}\int_{x \in M}$}] at (0.25,0) {};
  \end{tikzpicture}
\caption{$\sfA\sfA\sfA$ vertex}
\end{subfigure}
\begin{subfigure}{0.3\textwidth}
\centering
\begin{tikzpicture}[scale=2]
\node[shape=coordinate, label=below:{$\bbA$}] (O) at (0,0) {$
\bbA$};
\node[label=left:{$y$}] at (O.east) {$y$};
\node  at (0,0.5)  {}
edge[middlearrow={>}, "$i$" near start] (O);
  \tkzDrawPoint[color=black,fill=black,size=12](O)
  \node[shape=coordinate, label=right:{$\leadsto \varepsilon^{-1}\bbA^i(y)\int_{y \in \de M}$}] at (0.25,0) {};
\end{tikzpicture}
\caption{Univalent vertex}
\end{subfigure}
\caption{Feynman graphs and rules on the solid torus: Vertices}\label{fig:vertices}
\end{figure}

A Feynman graph $\Gamma$ is evaluated as follows. First, choose dual bases $\xi_i$ of $V$ and $\xi^i$ of $W$ and expand the fields accordingly: $\sfA = \sfA^i\xi_i = \bbA^i\xi_i + \sfa^i\xi_i + \alpha^i\xi_i$, and similarly for the $\sfB$ fields. Now one labels every trivalent vertex by a point $x \in M$, every univalent vertex by a point $y \in \de M$ (for this reason trivalent vertices are called bulk vertices and univalent vertices are called boundary vertices), and all half-edges with an index $i_k$. Every half-edge also ``inherits'' the label of the vertex it is attached to. Now one computes a de Rham current as follows. For every trivalent vertex with half-edges labeled  $(i,j,k)$ one multiplies with the appropriate structure constants, see Figure \ref{fig:vertices}. If a half-edge labeled by $(x,i)$ ends in the residual field $\sfa$, one multiplies with $\sfa^i(x)$, and similarly for the other half-edges and the $\sfb$ residual fields. For an edge formed by half-edges $(x,i)$ and $(y,j)$  multiply\footnote{$\eta$ is a de Rham current and one must be very careful taking products, see the paragraph below.} with $\delta^i_j\eta(x,y)$, see Figure \ref{fig:edges}.  Finally, for a boundary vertex labeled by $(y,i)$ multiply by $\bbA^i$. \\
As for power counting, vertices come with a factor of $\varepsilon^{-1} = \frac{i}{\hbar}$, while edges carry a factor of $\varepsilon = (-i\hbar)$. 
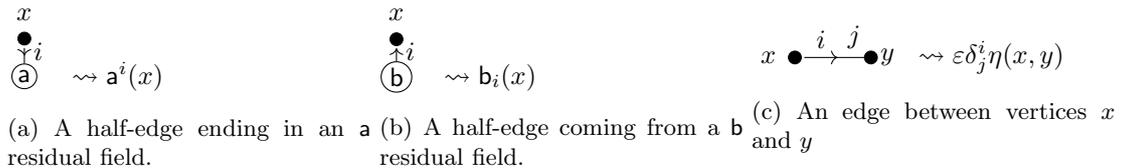
\begin{figure}[h]
\centering
\begin{subfigure}{0.3\textwidth}
\begin{tikzpicture}[scale=2]
\node[shape=circle, draw, inner sep=1pt] (O) at (0,0) {$
\sfa$};
\node[label=above:{$x$}] (v) at (0,0.25)  {} edge[middlearrow={>}, "$i$" near start] (O);
 \node[shape=coordinate, label=right:{$\leadsto \sfa^i(x)$}] at (0.25,0) {};
 \tkzDrawPoint[color=black,fill=black,size=12](v)
\end{tikzpicture}
\caption{A half-edge ending in an $\sfa$ residual field.}
\end{subfigure}
\begin{subfigure}{0.3\textwidth}
\begin{tikzpicture}[scale=2]
\node[shape=circle, draw, inner sep=1pt] (O) at (0,0) {$
\sfb$};
\node[label=above:{$x$}] (v) at (0,0.25)  {} edge[middlearrow={<}, "$i$" near start] (O);
 \node[shape=coordinate, label=right:{$\leadsto \sfb_i(x)$}] at (0.25,0) {};
 \tkzDrawPoint[color=black,fill=black,size=12](v)
\end{tikzpicture}
\caption{A half-edge coming from a $\sfb$ residual field.}
\end{subfigure}
\begin{subfigure}{0.3\textwidth}
\begin{tikzpicture}[scale=2]
\node[shape=coordinate, label=right:{$y$}] (O) at (0,0) {$
\sfa$};
\node[label=left:{$x$}] (v) at (-0.5,0)  {} edge[middlearrow={>}, "$i$" near start, "$j$" near end] (O);
 \node[shape=coordinate, label=right:{$\leadsto \varepsilon\delta^i_j\eta(x,y)$}] at (0.25,0) {};
 \tkzDrawPoint[color=black,fill=black,size=12](v)
 \tkzDrawPoint[color=black,fill=black,size=12](O)
\end{tikzpicture}
\caption{An edge between vertices $x$ and $y$}
\end{subfigure}
\caption{Feynman graphs and rules on the solid torus: Decorations and edges}\label{fig:edges}
\end{figure} In the next section we will see plenty of examples of Feynman graphs. There we depict the torus in a cross-section, placing boundary (univalent) vertices on the boundary and bulk (trivalent) vertices in the bulk. 
Notice that in the special case of the axial gauge, the weights will factor into integrals over the disk and integrals over the circle, since the residual fields and the propagator are products of forms on the disk and the circle. 
\subsubsection{Regularization}
Notice that since we work with the horizontal propagator, the integrals over the circle now contain distributions which cannot be extended to smooth forms on compactified configuration spaces as in \cite{Cattaneo2017}. Hence one needs another way to define the product of distributions arising in Feynman graphs. We will discuss this question, and the question of the correct space of states and operators to state the mQME, in more detail and generality elsewhere. A regularization that is conjecturally equivalent to the one used here, and based on approximating the axial gauge by bregular gauges, is discussed in the thesis of the third author 
\cite{Wernli2018}. 
For the purpose of this note is sufficient that to say that for tree diagrams the wavefront sets of the involved distributions are transversal and the product can be defined. The loop diagrams which are relevant for the computation of the Theta invariant will be discussed separately in section \ref{loops}. The choice of regularization is much easier for distributions on the circle, which is why we stick to the horizontal propagator. 
\section{Effective action on the solid torus}\label{sec:eff_action_torus}
We will now compute all the terms in the effective action contributing to  the two-loop contribution on a glued lens space. \\ The integrals over the bulk vertices will factorize into contributions from the disk and contributions from the circle, both of which can be computed explicitly. 
The results are presented in the following way: The contribution of a graph $\Gamma$ with $m$ univalent boundary vertices to the state is a sum of functionals on boundary fields of the form  
$$\ \psi_{\Gamma}(\bbA) = \int_{(\de M)^m}\omega_{\Gamma}\bbA_1\cdots\bbA_m = \sum_k P^k_{\Gamma}(z,z^+)_{i_1i_2\ldots i_m}\int_{(\de M)^m}\omega^k_{\Gamma}\bbA^{i_1}_1\cdots\bbA^{i_m}_m 
$$ 
where \begin{itemize}
\item $P^k_{\Gamma}(z,z^+)_{i_1\ldots i_m}$ is a product of $z$ and $z^+$ and the structure constants of the Lie algebra obtained by contracting the structure constants according to the graph, with $i_1,\ldots,i_m$ labeling the legs ending on the boundary, $k$ labels the different products of $z,z^+$ appearing, 
\item $\omega^k_{\Gamma} \colon = \pi_*(\widetilde{\omega}^k_{\Gamma})$ is a distributional form (that we call ``coefficient'') obtained by integrating the product of propagators and representatives of cohomology over the bulk points in the graph, i.e. taking the pushforward along the fibers of the map $M^{\times N} \times (\de M)^{\times m} \to (\de M)^{\times m}$ that forgets the bulk points,   

\item $\bbA^j_i$ denotes the pullback of the $\xi_j$ component of $\bbA$  under the $i$-th projection $(\de M)^m \to \de M$. 
\end{itemize} 
If we need to address individual coefficients we will usually distinguish them by form degree and denote a coefficient of form degree $p$ by $\omega^{(p)}_{\Gamma}$.
\subsection{Zero-point contribution}
By our choice of polarization, the zero-point contribution - the effective action of the free part - just consists of the term 
$$\psi_{\Gamma_0}=  -\varepsilon^{-1} \int_{\de_1M} \sfb_k\bbA^k = \varepsilon^{-1}\left(-z_{1,k}^+\int_{\de_1M}\bbA^k -z_{2,k}^+\int_{\de_1M}dt\bbA^k\right). $$
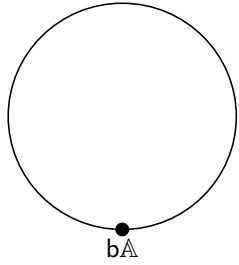
\begin{figure}[h]
\begin{tikzpicture}[scale=1]

  \node[shape=coordinate] (O) at (0,0) {};

  \node[shape=coordinate,label=below:{$\sfb\bbA$}] (bdry1) at (0,-1.5) {$\sfb\bbA$};
\tkzDrawCircle(O,bdry1)
 \tkzDrawPoints[color=black,fill=black,size=12](bdry1)
\end{tikzpicture}
\caption{Single diagram contributing to zero-point effective action}
\end{figure}
\subsection{One-point contribution} 
\subsubsection{Possible diagrams}
If we have just one bulk vertex it can carry no $\beta$'s due to the polarization. Also, notice that $\sfa^2$=0. The following labelings survive:
$\sfa\sfb\sfb,\sfb\sfb\sfb, \sfa\sfb\alpha, \sfb\sfb\alpha, \sfa\alpha\alpha,\sfb\alpha\alpha,\alpha\alpha\alpha$. One can check that upon integration the ones with two $\sfb$'s vanish for degree reasons, with the exception of $\sfa\sfb\sfb$. Let us look at the remaining terms. 
\begin{figure}[h]
\centering
\begin{subfigure}{0.3\textwidth}
\centering
\begin{tikzpicture}[scale=1]

  \node[shape=coordinate] (O) at (0,0) {};

 \tkzDefPoint(1.5,0){A}
  \node (resb) at (0,0.5) [shape=circle,above,draw,minimum size=1pt] {$\sfa$}
edge[middlearrow={<}] (O);
\node (resa) at (-0.3,-0.4) [shape=circle,below left,draw,minimum size=1pt] {$\sfb$}
edge[middlearrow={>}] (O);
\node (resa) at (0.3,-0.4) [shape=circle,below right,draw,minimum size=1pt] {$\sfb$}
edge[middlearrow={>}] (O);
  \tkzDrawCircle(O,A)
  \tkzDrawPoint[color=black,fill=black,size=12](O)
\end{tikzpicture}
\caption{$\Gamma_{1,0}$}
\end{subfigure}
\begin{subfigure}{0.3\textwidth}
\centering
\begin{tikzpicture}[scale=1]

  \node[shape=coordinate] (O) at (0,0) {};

  \node[shape=coordinate,label=below:{$\bbA$}] (bdry1) at (0,-1.5) {$\bbA$}
  edge[<-,shorten <= 1.5pt,thick] (O);
\node (resb) at (-0.5,0.5) [shape=circle,above,draw,minimum size=1pt] {$\sfb$}
edge[middlearrow={>}] (O);
\node (resa) at (0.5,0.5) [shape=circle,above,draw,minimum size=1pt] {$\sfa$}
edge[middlearrow={<}] (O);
  \tkzDrawCircle(O,A)
  \tkzDrawPoints[color=black,fill=black,size=12](O,bdry1)
\end{tikzpicture}
\caption{$\Gamma_{1,1}$}
\end{subfigure}
\begin{subfigure}{0.3\textwidth}
\centering
\begin{tikzpicture}[scale=1.5]
\node[shape=coordinate] (O) at (0,0) {};
\node (resb) at (0,0.3) [shape=circle,above,draw,minimum size=1pt] {$\sfb$}
edge[middlearrow={>}] (O);
\node[shape=coordinate,label=below:{$\bbA$}] (bdry1) at (canvas polar cs: angle=-120, radius=1cm) {$\bbA$}
edge[<-,shorten <= 1.5pt,thick] (O);
\node[shape=coordinate,label=below:{$\bbA$}] (bdry2) at (canvas polar cs: angle=-60, radius=1cm) {$\bbA$}
edge[<-,shorten <= 1.5pt,thick] (O);

\tkzDrawCircle(O,bdry1)
  \tkzDrawPoints[color=black,fill=black,size=12](O,bdry1,bdry2)

\end{tikzpicture}
\caption{$\Gamma_{1,2}^{\sfb}$}
\end{subfigure}
\begin{subfigure}{0.3\textwidth}
\centering
\begin{tikzpicture}[scale=1.5]
\node[shape=coordinate] (O) at (0,0) {};
\node (resb) at (0,0.3) [shape=circle,above,draw,minimum size=1pt] {$\sfa$}
edge[middlearrow={>}] (O);
\node[shape=coordinate,label=below:{$\bbA$}] (bdry1) at (canvas polar cs: angle=-120, radius=1cm) {$\bbA$}
edge[<-,shorten <= 1.5pt,thick] (O);
\node[shape=coordinate,label=below:{$\bbA$}] (bdry2) at (canvas polar cs: angle=-60, radius=1cm) {$\bbA$}
edge[<-,shorten <= 1.5pt,thick] (O);

\tkzDrawCircle(O,bdry1)
  \tkzDrawPoints[color=black,fill=black,size=12](O,bdry1,bdry2)

\end{tikzpicture}
\caption{$\Gamma_{1,2}^{\sfa}$}
\end{subfigure}
\begin{subfigure}{0.3\textwidth}
\centering
\begin{tikzpicture}[scale=1.5]
\node[shape=coordinate] (O) at (0,0) {};
\node[shape=coordinate,label=below:{$\bbA$}] (bdry1) at (canvas polar cs: angle=-130, radius=1cm) {$\bbA$}
edge[<-,shorten <= 1.5pt,thick] (O);
\node[shape=coordinate,label=below:{$\bbA$}] (bdry2) at (canvas polar cs: angle=-90, radius=1cm) {$\bbA$}
edge[<-,shorten <= 1.5pt,thick] (O);

\node[shape=coordinate,label=below:{$\bbA$}] (bdry3) at (canvas polar cs: angle=-50, radius=1cm) {$\bbA$}
edge[<-,shorten <= 1.5pt,thick] (O);

\tkzDrawCircle(O,bdry1)
  \tkzDrawPoints[color=black,fill=black,size=12](O,bdry1,bdry2,bdry3)

\end{tikzpicture}
\caption{$\Gamma_{1,3}$}
\end{subfigure}
\caption{Graphs in the solid torus (depicted in a cross-section) with 1 interaction vertex. A bullet denotes a point we integrate over, a long arrow denotes a propagator.}
\end{figure}
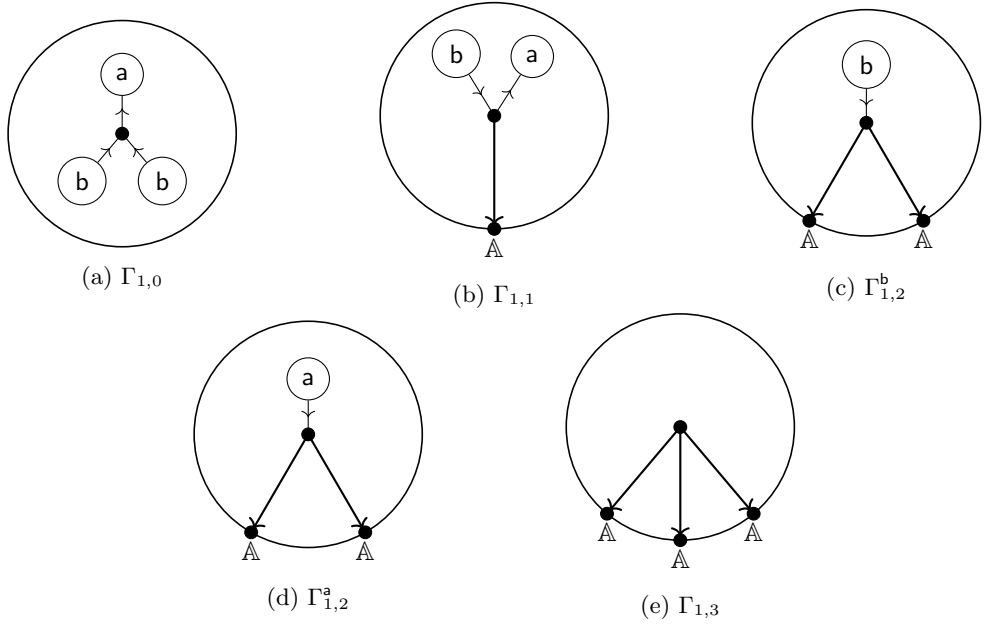
\subsubsection{$\sfb\sfa\sfa$ term}
The first one contains no propagator: 
$$\Psi_{\Gamma_{1,0}} =\frac{\varepsilon^{-1}}{2}\int_M\langle \sfa, [\sfb,\sfb]\rangle = \frac{\varepsilon^{-1}}{2}g_i^{jk}(z^{1i}z^+_{1j}z^+_{1k}+2z^{2i}z^+_{1j}z^+_{2k}) = \varepsilon^{-1}(P^1_{\Gamma_{1,0}} + P^2_{\Gamma_{1,0}})$$ 
For the other terms explicit computations of the pushforwards over bulk vertices can be found in the thesis of the third author (\cite{Wernli2018}). 
\subsubsection{$\sfb\sfa\alpha$ term} We have 
\begin{align*}
\Psi_{\Gamma_{1,1}} :&= - \varepsilon^{-1}\int_{M \times \de M}  g^i_{jk} \sfb_{1,i}\sfa_1^j\eta_{12}\bbA_2^k \\
&= -\varepsilon^{-1}g^i_{jk} (z_{1i}^+z^{1j} - z_{2i}^+z^{2j})\int_{\de M}\pi_*\left(\mu_1dt_1\eta_{12}\right))\bbA^k_2 - \varepsilon^{-1}g^i_{jk}z_{1i}^+z^{2j}\int_{\de M}\pi_*\left(\mu_1\eta_{12}\right)\bbA^k_2.
\end{align*}
\subsubsection{$\sfb\alpha\alpha$ term}
Let us turn to the next term 
\begin{align*}\Psi_{\Gamma_{1,2}^{\sfb}}:&= \frac{\varepsilon^{-1}}{2}\int_{M \times \de M \times \de M}g^i_{jk}b_{1,i}\eta_{12}\eta_{13}\bbA_2^j\bbA_3^k 
\end{align*}It evaluates to
\begin{align*}
\Psi_{\Gamma_{1,2}^{\sfb}} &=  \frac{\varepsilon^{-1}}{2}g^i_{jk}z_{1i}^+\int_{ \de M \times \de M} \pi_*\left(\eta_{12}\eta_{13}\right)\bbA_2^j\bbA_3^k + \frac{\varepsilon^{-1}}{2}g^i_{jk} z_{2i}^+\int_{ \de M \times \de M} \pi_*\left(dt_1\eta_{12}\eta_{13} \right)\bbA_2^j\bbA_3^k
\end{align*}
\subsubsection{$\sfa\alpha\alpha$ term}
The next term is 
\begin{align*}\Psi_{\Gamma_{1,2}^{\sfa}}:&= \frac{\varepsilon^{-1}}{2}\int_{M  \times \de M \times \de M}f_{ijk}a_1^i\eta_{12}\eta_{13}\bbA_2^j\bbA_3^k
\end{align*}
It evaluates to 
\begin{align*}
\Psi_{\Gamma_{1,2}^{\sfa}} &=  -\frac{\varepsilon^{-1}}{2}f_{ijk}z^{2,i}\int_{ \de M \times \de M} \pi_*\left(\mu_1\eta_{12}\eta_{13}\right)\bbA_2^j\bbA_3^k - \frac{1}{2}f_{ijk}z^{1,i}\int_{ \de M \times \de M} \pi_*\left(\mu_1dt_1\eta_{12}\eta_{13}\right)\bbA_2^j\bbA_3^k
\end{align*}
\subsubsection{$\alpha\alpha\alpha$ term}
The next term is 
\begin{align*}\Psi_{\Gamma_{1,3}}:&= \frac{1}{6}\int_{M \times \de M \times \de M \times \de M}f_{ijk}\eta_{12}\eta_{13}\eta_{14}\bbA_2^i\bbA_3^j\bbA_4^k 
\end{align*}
It evaluates to 
\begin{align*}
\Psi_{\Gamma_{1,3}} =   -\frac{1}{6}f_{ijk}\int_{ \de M \times \de M\times \de M}
\pi_*\left(\eta_{12}\eta_{13}\eta_{14}\right)\bbA_2^i\bbA_3^j\bbA_4^k 
\end{align*}
Notice the last two terms are not present in the case where $(\g,V,W)$ is a Manin triple. 
\subsection{2-point tree contribution}\label{sec:2pttrees}
In principle, tree diagrams with two points could contribute to the 2-point effective action after gluing. However, in this subsection we will argue that it is not so. The point is that in the gluing process 2-point graphs can only be paired with 0-point graphs on the other side, i.e. graphs from the free effective action. As explained below in section \ref{sec:order0pairing}, pairing against a 0-point diagram on the other side amounts to placing a linear combination of $1$, $dt$ and $d\theta$ at this point and integrating over it. The claim then follows from the following lemmata: 
\begin{lem}\label{lem:ResFieldBdryIntegration}
For $x_2 \in \de M$, we have 
\begin{align*}
\int_2 \eta_{12}dt_2 &= dt_1 \\
\int_2\eta_{12}d\theta_2 &= -\psi = \frac{z d\bar{z}-\bar{z}dz}{4\pi i}. 
\end{align*}
\end{lem}
\begin{proof}
We have 
\begin{align*}
\int_2 \eta_2dt_2 &= \int_2(\eta_{D,12}\delta_{S^1,12}+\mu_1\eta_{S^1,12})dt_2 = 
\int_{\de D,2}\eta_{D,12}\int_{S^1,2}\delta_{S^1,12}dt_2 = dt_1
\end{align*}
since $\int_{\de D,2}\eta_{D,12} = 1 $ and 
\begin{align*}
\int_2 \eta_2d\theta_2 &= (\eta_{D,12}\delta_{S^1,12}+\mu_1\eta_{S^1,12})d\theta_2 = \int_{\de D,2}\eta_{D,12}d\theta_2 \\ 
&= \int_{\de D,2} 2(\phi_{12}-\psi_1)d\theta_2 = -\psi_1
\end{align*}
since $\int_{\de D,2} \phi_{12} d\theta_2 = 0$, as can easily be checked by the residue theorem. 
\end{proof}
\begin{lem}
Integrating $\eta$ against $dt,\psi$ or $\psi dt$ placed either at head or boundary vanishes. 
\end{lem}
\begin{proof}
This follows from the fact that $\eta_{S^1}$ (resp. $\eta_D$) vanish when integrated against $dt$ (resp. $\psi$). 
\end{proof}
Now consider  a two-point tree diagram as in figure \ref{2pttree}. It consists of a single arrow between the two points, some legs on the boundary, and maybe some residual fields. Now integrate all the legs against $dt$ or $d\theta$. The result is a graph consisting of single arrow with a product $\gamma_i$ of residual fields, $dt$'s and $d\theta$'s on both ends (figure \ref{2pttree2}). From the two Lemmata above together with Proposition \ref{prop:propprops}  it now follows that the contribution of such a graph is zero after gluing. 

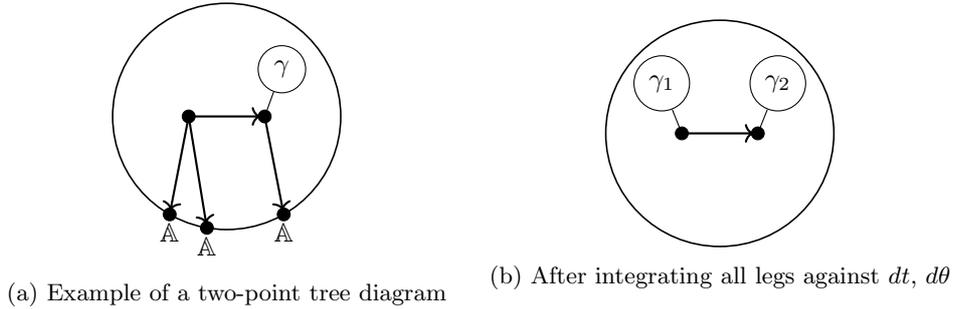
\begin{figure}[h]
\begin{subfigure}{0.4\textwidth}
\centering
\begin{tikzpicture}[scale=1]
\coordinate (O) at (0,0);
 \coordinate (bulk1) at (-0.5,0) {};
 \coordinate (bulk2) at (0.5,0.0) {}
 edge[<-, shorten <= 1.5pt,thick] (bulk1);
 \node[shape=coordinate,label=below:{$\bbA$}] (bdry1) at (canvas polar cs: angle=-120, radius=1.5cm) {$\bbA$}
edge[<-,shorten <= 1.5pt,thick] (bulk1);
 \node[shape=coordinate,label=below:{$\bbA$}] (bdry2) at (canvas polar cs: angle=-100, radius=1.5cm) {$\bbA$}
edge[<-,shorten <= 1.5pt,thick] (bulk1);
\node[shape=coordinate,label=below:{$\bbA$}] (bdry3) at (canvas polar cs: angle=-60, radius=1.5cm) {$\bbA$}
edge[<-,shorten <= 1.5pt,thick] (bulk2);
\node (resa) at (0.5,0.4) [shape=circle,above right,draw] {$\gamma$}
edge(bulk2);
\tkzDrawCircle(O,bdry1)
  \tkzDrawPoints[color=black,fill=black,size=12](bulk1,bulk2,bdry1,bdry2,bdry3)
\end{tikzpicture}
\caption{Example of a two-point tree diagram}\label{2pttree}
\end{subfigure}
\begin{subfigure}{0.4\textwidth}
\centering
\begin{tikzpicture}[scale=1]
\coordinate (O) at (0,0);
 \coordinate (bulk1) at (-0.5,0) {};
 \coordinate (bulk2) at (0.5,0.0) {}
 edge[<-, shorten <= 1.5pt,thick] (bulk1);
\node (bdry1) at (canvas polar cs: angle=-60, radius=1.5cm) {};
\node (resb) at (-0.5,0.4) [shape=circle,above left,draw] {$\gamma_1$}
edge (bulk1);
\node (resa2) at (0.5,0.4) [shape=circle,above right,draw] {$\gamma_2$}
edge (bulk2);
\tkzDrawCircle(O,bdry1)
  \tkzDrawPoints[color=black,fill=black,size=12](bulk1,bulk2)
\end{tikzpicture}
\caption{After integrating all legs against $dt$, $d\theta$}\label{2pttree2}
\end{subfigure}
\caption{Two-point tree diagrams}
\end{figure}

\subsection{Loop diagrams}\label{loops} 
At two-point order one can also see loops appearing. Loop diagrams will contain products of distributions that cannot be defined without further choices, namely, the choice of a regularization procedure. Usually one extends the propagator smoothly to a compactification of the configuration space, this can be interpreted as a sort of point-splitting regularization. For the distributional propagators at hand such a regularization is not possible, and another normalization scheme is needed.
\subsubsection{Regularization} For the purpose of this note, where we work with distributional forms on the circle, the choice of regularization is straightforward. Consider the loop given by two arrows between the two points (their directions do not matter, since the circle delta form and propagator are symmetric up to a sign). The ill-defined products of distributions that can appear are $$ \left(\delta^{(1)}_{S^1}(t_1 - t_2)\right)^2 \qquad \text{ and } \qquad \delta^{(1)}_{S^1}(t_1 - t_2)\eta_{S^1}(t_1,t_2).$$ 
The only sensible value to attach to these two terms is zero\footnote{In the thesis of the third author \cite{Wernli2018}, this is called the ``universal'' regularization of the axial gauge. Conjecturally, it can be obtained by approximating the axial gauge via Riemann-Hodge gauges.}. The first one is the square of the  one-form  $\delta^{(1)}_{S^1}(t_1 - t_2) = \delta(t_1-t_2)(dt_1-dt_2)$. For the second one, notice that the circle propagator is antisymmetric with respect to the diagonal, so its value there should be zero\footnote{Here we are thinking of the propagator as a function on $S^1 \times S^1$, so we should define it on the diagonal.}. A slightly more rigorous approach is to smear out the delta function, which will give the above results, as long as one chooses a symmetric nascent delta function. Setting the circle propagator to 0 on the diagonal means that the product with the delta distribution vanishes. These two choices are enough to evaluate all the loop diagrams at two-point order, and, in fact, at any order. 
\subsubsection{Evaluation}
Consider first a loop with both arrows pointing the same way as in figure \ref{fig:loopsame}, call this a loop of type A. The contribution of such a loop is 
\begin{align*}
(\eta_{12})^2 &= \left(\eta_{D,12}\delta^{(1)}_{S^1,12} + \mu_1\eta_{S^1,12}\right)^2 = \left(\eta_{D,12}\delta^{(1)}_{S^1,12}\right)^2 + 2 \mu_1\eta_{D,12}\delta_{S^1,12}^{(1)} + (\mu_1\eta_{S^1,12})^2 = 0
\end{align*}
by our choice of regularization above, and since $\mu^2 =0$. On the other hand, if the arrows point in opposite directions as in figure \ref{fig:loopopposite} (call this loop of type B), the loop contributes 
\begin{align*}
\eta_{12}\eta_{21} = \left(\eta_{D,12}\delta^{(1)}_{S^1,12} + \mu_1\eta_{S^1,12}\right)\left(\eta_{D,21}\delta^{(1)}_{S^1,12} + \mu_2\eta_{S^1,21}\right) = - \mu_1\mu_2\eta_{S^1,12}^2
\end{align*}
and again all other terms vanish due to the regularization. 
\begin{figure}[h]
\centering
\begin{subfigure}{0.4\textwidth}
\centering
\begin{tikzpicture}[scale=1]
\coordinate (O) at (0,0);
 \coordinate (bulk1) at (-0.75,0) {};
 \coordinate (bulk2) at (0.75,0.0) {};
 \draw[->,thick,-{Latex}] (bulk1) to[out=60,in=120] (bulk2);
 \draw[->,thick,-{Latex}] (bulk1) to[out=-60,in=-120] (bulk2);
\draw (O) circle [radius=1.5cm];
  \tkzDrawPoints[color=black,fill=black,size=12](bulk1,bulk2);
\end{tikzpicture}
\caption{Loop with both arrows directed the same way (``type A'')}\label{fig:loopsame}
\end{subfigure}
\begin{subfigure}{0.4\textwidth}
\centering
\begin{tikzpicture}[scale=1]
\coordinate (O) at (0,0);
 \coordinate (bulk1) at (-0.75,0) {};
 \coordinate (bulk2) at (0.75,0.0) {};
 \draw[->,thick,-{Latex}] (bulk1) to[out=60,in=120] (bulk2);
 \draw[->,thick,-{Latex}] (bulk2) to[out=-120,in=-60] (bulk1);
\draw (O) circle [radius=1.5cm];
  \tkzDrawPoints[color=black,fill=black,size=12](bulk1,bulk2);
\end{tikzpicture}
\caption{Loop with both arrows directed opposite ways ( ``type B'')}\label{fig:loopopposite}
\end{subfigure}
\end{figure}
In particular, with this regularization the contribution of the theta graph is zero, since it always contains a type A loop. From the contribution of the type B loop we see that if we place an $\sfa$ residual field at either end the contribution vanishes (since $\sfa \wedge \mu = 0$). So, the contributing loop diagrams contain a type B loop with a $\sfb$ residual or propagator placed at the ends. The resulting loop diagrams can be seen in figure \ref{fig:loopdiagrams}. \\
\begin{figure}[!h]
\centering
\begin{subfigure}{0.3\textwidth}
\centering
\begin{tikzpicture}[scale=1]
\coordinate (O) at (0,0);
 \coordinate (bulk1) at (-0.75,0) {};
 \coordinate (bulk2) at (0.75,0.0) {};
 \node[shape=circle,draw,above left] (res1) at (-0.6,0.4) {$\sfb$}
 edge[middlearrow={>}] (bulk1);
 \node[shape=circle,draw,above right] (res2) at (0.6,0.4) {$\sfb$}
 edge[middlearrow={>}] (bulk2);
 \draw[->,thick,-{Latex}] (bulk1) to[out=60,in=120] (bulk2);
 \draw[->,thick,-{Latex}] (bulk2) to[out=-120,in=-60] (bulk1);
\draw (O) circle [radius=1.5cm];
  \tkzDrawPoints[color=black,fill=black,size=12](bulk1,bulk2);
\end{tikzpicture}
\caption{$\Gamma_{2,0}$}\label{fig:loop1}
\end{subfigure}
\begin{subfigure}{0.3\textwidth}
\centering
\begin{tikzpicture}[scale=1]
\coordinate (O) at (0,0);
 \coordinate (bulk1) at (-0.75,0) {};
 \coordinate (bulk2) at (0.75,0.0) {};
 \node[shape=circle,draw,above left] (res1) at (-0.6,0.4) {$\sfb$}
 edge[middlearrow={>}] (bulk1);
 \node[shape=coordinate,draw,label=below:{$\bbA$}] (bdry1) at (canvas polar cs: angle=-60, radius=1.5cm) {$\bbA$} 
 edge[thick,{Latex}-] (bulk2);
 \draw[thick,-{Latex}] (bulk1) to[out=60,in=120] (bulk2);
 \draw[thick,-{Latex}] (bulk2) to[out=-120,in=-60] (bulk1);
\draw (O) circle [radius=1.5cm];
  \tkzDrawPoints[color=black,fill=black,size=12](bulk1,bulk2,bdry1);
\end{tikzpicture}
\caption{$\Gamma_{2,1}$}\label{fig:loop2}
\end{subfigure}
\begin{subfigure}{0.3\textwidth}
\centering
\begin{tikzpicture}[scale=1]
\coordinate (O) at (0,0);
 \coordinate (bulk1) at (-0.75,0) {};
 \coordinate (bulk2) at (0.75,0.0) {};
 \draw[->,thick,-{Latex}] (bulk1) to[out=60,in=120] (bulk2);
 \draw[->,thick,-{Latex}] (bulk2) to[out=-120,in=-60] (bulk1);
\draw (O) circle [radius=1.5cm];
\node[shape=coordinate,draw,label=below:{$\bbA$}] (bdry1) at (canvas polar cs: angle=-60, radius=1.5cm) {$\bbA$} 
 edge[thick,{Latex}-] (bulk2);
\node[shape=coordinate,draw,label=below:{$\bbA$}] (bdry2) at (canvas polar cs: angle=-120, radius=1.5cm) {$\bbA$} 
 edge[thick,{Latex}-] (bulk1);
  \tkzDrawPoints[color=black,fill=black,size=12](bulk1,bulk2,bdry1,bdry2);
\end{tikzpicture}
\caption{$\Gamma_{2,2}$}\label{fig:loop3}
\end{subfigure}
\caption{Loop diagrams in the 2-point effective action}\label{fig:loopdiagrams}
\end{figure}
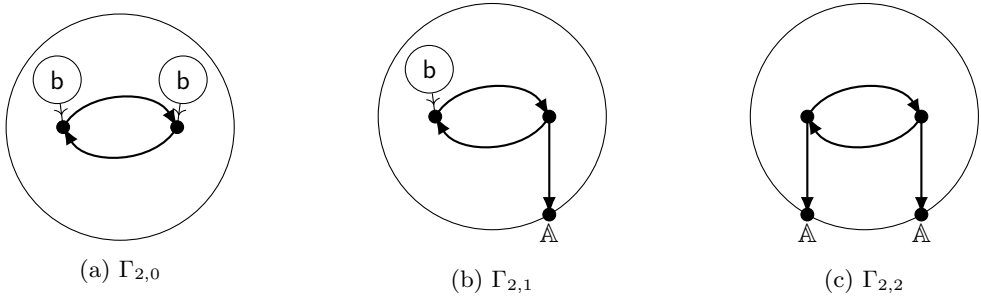

Their contributions can then be evaluated to 
\begin{align}
\psi_{\Gamma_{2,0}} &=  \frac{\varepsilon}{2}z_{2,i}^+z_{2,j}^+h^{ik}_lh^{jl}_k\pi_*(-\mu_1dt_1\mu_2dt_2\eta^2_{S^1,12}) \\
\psi_{\Gamma_{2,1}} &=  \frac{\varepsilon}{12}z_{2,i}^+h^{ij}_kg^k_{jl}\int_{\de M,1}\pi_*(\mu_1dt_1\mu_2\eta_{23}) \bbA^l_3 \\ 
\psi_{\Gamma_{2,2}} &= \frac{\varepsilon}{2}g^{k}_{ij}g^j_{kl}\int_{\de M \times \de M}\pi_*(\mu_1\mu_2\eta_{S^1,12}^2\eta_{13}\eta_{24})\bbA^i_3\bbA^l_4
\end{align}
\begin{rem}[Regularization and framing contribution] 
One can say that imposing $(\delta^{(1)}_{S^1} = 0)^2$ (alternatively, using the axial gauge) is using the blackboard framing of the solid torus. As we will discuss below, the different choices of gluing diffeomorphism leading to the same diffeomorphism type of lens spaces are related by changing the framing on one of two solid tori, and this change is reflected through a changed value of the two-loop contribution. Thus one can argue that the regularization of the theory depends on the choice of a framing, an effect that is visible in other approaches to perturbative Chern-Simons theory: In the Axelrod-Singer approach one cancels dependence on the regularization through the introduction of a framing, in the Kontsevich-Kuperberg-Thurston approach the framing defines the regularization. 
\end{rem}
\section{Gluing of lens spaces}\label{sec:Gluing}
In this section we will describe how to compute the 2-point effective action on lens spaces by applying the gluing procedure described in \cite{Cattaneo2017}. The computation is carried out in the next section.
\subsection{Lens spaces}
In this note we will apply the following conventions for lens spaces. 
Consider two solid tori $M_1 = S^1 \times D = M_2$. The boundary is $S^1 \times S^1$ with coordinates $(t,\theta) \in (\R/\Z)^2$. Pick two coprime integers $p$ and $q$. Since they are  coprime, there exist $m,n$ such that $mq-np =1$. Let $\varphi \in \mathrm{Diff}(S^1\times S^1)$ be defined by 
\begin{equation}
\label{def:varphidiffeo}
 \varphi \colon \begin{pmatrix} t \\ \theta  \end{pmatrix} \mapsto \begin{pmatrix} m & p \\ n & q \end{pmatrix} \begin{pmatrix} t \\ \theta  \end{pmatrix} = \begin{pmatrix} mt+p\theta \\ nt + q\theta  \end{pmatrix}
\end{equation} Then we define the lens space $L_{p,q}$ by 
\begin{equation}
L_{p,q} = M_1 \cup_{\varphi} M_2.\label{def:Lpq}
\end{equation}
Note that with this convention $L_{1,0} \cong S^3$ and $L_{0,1} \cong S^1 \times S^2$. 
\begin{rem}[Dependence of framing on choices]\label{rem:framinggluing}
It is well known that the diffeomorphism type of $L_{p,q}$ is independent of the choice of $m$ and $n$ and also independent of the choice of $q$ (mod $p$). Changing $q$ by a multiple of $p$ , to preserve the equation $mq-np=1$, we have to change $n$ by the same multiple of $m$.
Let 
$$ T = \begin{pmatrix} 1 & 0 \\ 1 & 1 \end{pmatrix}.$$ Then the former change corresponds to multiplying $\varphi$ with $T^k$ from the right, while the latter corresponds to multiplying with $T^k$ from the left. Since in the gluing we identify $\de M_1 \ni x \sim \varphi(x) \in \de M_2$, multiplying from the right with $T^k$ corresponds to gluing after performing $k$ Dehn twists around the longitude (given by $(t,0)$) in $\de M_2$, while the latter operation corresponds to performing $k$ \emph{inverse} Dehn twists around the longitude in $\de M_1$. We can extend Dehn twists around the longitude to the bulk of the solid torus: Using polar coordinates on the disk a possible representation\footnote{Only the isotopy class of a Dehn twist is well-defined.} is\footnote{The following formula can be extended to $r=0$ by the identity.}
$$ T \colon (t,r,\theta) \mapsto (t,r, \theta + t).$$ A Dehn twist changes the homotopy class of the framing on the solid torus by one generator. Hence both operations, i.e. shifting either $m$ or $q$ by $p$ and performing the corresponding shift of $n$, correspond to changes of framing of the resulting lens space. The first operation will change the framing of $L_{p,q}$ by $+k$ units, the second operation by $-k$ units. This is discussed in detail in \cite[Appendix B]{Freed1991}.
\end{rem}
\begin{rem}[The case $p=0$]
If $p=0$, then $qm = 1$, so we have $q = m = \pm 1$. That case needs to be considered separately: The resulting space $S^1 \times S^2$ is not a  rational homology 3-sphere, unlike all lens spaces. In the following we will always assume $p \neq 0$ unless otherwise stated. 
\end{rem}
\subsection{Gluing perturbative expansions in BV-BFV}\label{sec:Glueing}
The gluing procedure discussed in \cite{Cattaneo2017} amounts to the following prescription: 
\begin{itemize}
\item Take a diagram $\Gamma_1$ in the $\bbA$-representation and a diagram $\Gamma_2$ in the $\bbB$-representation with the same number $n$ of legs and multiply $\psi_{\Gamma_1} = \int_{(\de M_1)^n}\omega_{\Gamma_1}\bbA_1\cdots\bbA_n$ and $\psi_{\Gamma_2} = \int_{(\de M_2)^n}\omega_{\Gamma_2}\bbB_1\cdots\bbB_n$. Here the diagrams can be non connected, since the state is the exponential of the effective action. 
\item Sum over all ways of contracting $\bbA$ and $\bbB$ fields to a delta form $\epsilon\delta^{(d-1)}_{\de M}(x,\varphi(x))$ (the factor $\varepsilon$ comes from the definition via path integrals as in \cite{Cattaneo2017}).
\item Perform the integration over $(\de M_1)^n \times (\de M_2)^n$.
\item Reduce the residual fields. 
\end{itemize}
Equivalently, the state glued from $\Psi_{\Gamma_1}$ and $\Psi_{\Gamma_2}$ can be defined as follows. Let $\sigma \in S_n$ be a permutation and denote $\Phi_{\sigma} \colon (\de M)^n \to (\de M)^n$ the map defined by 
$$(x_1,\ldots,x_n) \mapsto (\varphi(x_{\sigma(1)}),\ldots,\varphi(x_{\sigma(n)})).$$
Then the above prescription results in  
$$\psi_{\Gamma_1} * \psi_{\Gamma_2} := \sum_{\sigma \in S_n} \int_{(\de M)^n} \omega_{\Gamma_1} \Phi_{\sigma}^* \omega_{\Gamma_2}.$$
In this integral only the top degree part survives. Sometimes it will be convenient to use the reformulation
\begin{equation}
\label{eq:inversegluing}
\psi_{\Gamma_1} * \psi_{\Gamma_2} = \sum_{\sigma \in S_n} \int_{(\de M)^n}  (\Phi_{\sigma}^{-1})^*( \omega_{\Gamma_1} \Phi_{\sigma}^* \omega_{\Gamma_2}) = \sum_{\sigma \in S_n} \int_{(\de M)^n}  (\Phi_{\sigma}^{-1})^*( \omega_{\Gamma_1})  \omega_{\Gamma_2}. 
\end{equation}
In all examples that we consider, the graphs are invariant with respect to permuting the boundary points. Hence, the sum is a constant times the pairing computed using $\Phi_{\mathrm{id}}$, which we will also denote $\varphi$, abusing notation.  \\
Notice also that for graphs not depending on boundary fields the gluing procedure is trivial, i.e. the corresponding contributions are simply multiplied with the rest, or, equivalently, added to the effective action. \\
\subsection{The effective action on $M_2$}\label{sec:Brep}
On $M_2$ we will choose the opposite polarization, namely, $\de M = \de_2 M$ (the $\bbB$-representation). To avoid confusion, from now on  we decorate objects with a superscript depending on which representation they are computed in, e.g. $\psi^{\bbA}$ (resp. $\psi^{\bbB}$) denotes the state in the $\bbA$-representation ($\bbB$-representation). The residual fields change roles:  $\sfa^{\bbB} = z^{\bbB}_11 + z^{\bbB}_2dt$, $\sfb^{\bbB} = z^{+,\bbB}_1\mu dt + z^{+,\bbB}_2 \mu$. Now the self-duality of Chern-Simons theory comes in handy. Let 
$\Gamma^{\bbA}$ be a diagram appearing in the $\bbA$-representation. Then there is a diagram $\Gamma^{\bbB}$ in the $\bbB$-representation obtained from $\Gamma^{\bbA}$ by reversing all arrows and exchanging $\sfb^{\bbA}$ (resp. $\sfa^{\bbA}$) with $\sfa^{\bbB}$ (resp. $\sfb^{\bbB}$) fields (and of course $\bbA$ and $\bbB$ fields). See figure \ref{fig:CorrespondingDiagrams}. We will call this diagram $\Gamma^{\bbB}$ the \emph{dual diagram} of $\Gamma^{\bbA}$. 
\begin{figure}[h]
\begin{subfigure}{0.4\textwidth}
\centering
\begin{tikzpicture}[scale=1.5]
\node[shape=coordinate] (O) at (0,0) {};
\node (resb) at (0,0.3) [shape=circle,above,draw,minimum size=1pt] {$\sfb$}
edge[middlearrow={>}] (O);
\node[shape=coordinate,label=below:{$\bbA$}] (bdry1) at (canvas polar cs: angle=-120, radius=1cm) {$\bbA$}
edge[<-,{Latex}-,thick] (O);
\node[shape=coordinate,label=below:{$\bbA$}] (bdry2) at (canvas polar cs: angle=-60, radius=1cm) {$\bbA$}
edge[<-,{Latex}-,thick] (O);
\tkzDrawCircle(O,bdry1)
  \tkzDrawPoints[color=black,fill=black,size=12](O,bdry1,bdry2)
\end{tikzpicture}
\caption{ $\Gamma^{\bbA}$ }
\end{subfigure}
\begin{subfigure}{0.4\textwidth}
\centering
\begin{tikzpicture}[scale=1.5]
\node[shape=coordinate] (O) at (0,0) {};
\node (resb) at (0,0.3) [shape=circle,above,draw,minimum size=1pt] {$\sfa$}
edge[middlearrow={<}] (O);
\node[shape=coordinate,label=below:{$\bbB$}] (bdry1) at (canvas polar cs: angle=-120, radius=1cm) {$\bbB$}
edge[-{Latex},thick] (O);
\node[shape=coordinate,label=below:{$\bbB$}] (bdry2) at (canvas polar cs: angle=-60, radius=1cm) {$\bbB$}
edge[-{Latex},thick] (O);
\tkzDrawCircle(O,bdry1)
  \tkzDrawPoints[color=black,fill=black,size=12](O,bdry1,bdry2)
\end{tikzpicture}
\caption{Corresponding  $\Gamma^{\bbB}$ }
\end{subfigure}
\caption{Corresponding diagrams in $\bbA$- and in $\bbB$-representations.}\label{fig:CorrespondingDiagrams}
\end{figure}
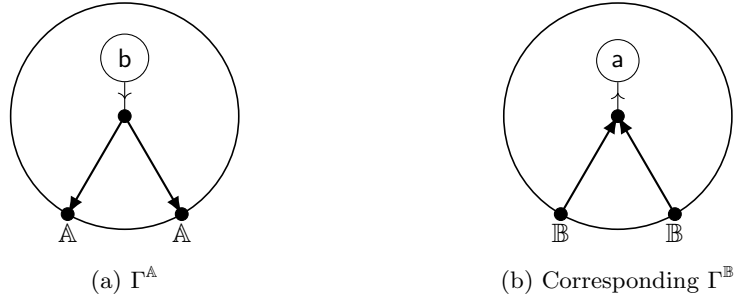
However, since we have 
$$ \eta^{\bbA}(x_1,x_2) = - \eta^{\bbB}(x_2,x_1),$$
and the form parts of $\sfa^{\bbA}$ and $\sfb^{\bbB}$ (resp. $\sfa^{\bbB}$ and $\sfb^{\bbA}$) are the same, the form parts of $\Gamma^{\bbA}$ and $\Gamma^{\bbB}$ are the same up to a sign. The structure constants and coordinates on the space of residual fields have to replaced with their ``dual'' counterparts. This amounts to the following prescription. To compute the state $\psi_{\Gamma^\bbB}$ from the state $\psi_{\Gamma^\bbA}$, 
\begin{itemize} 
\item replace $z^{i,k,\bbA}$ by $z_{i,k}^{+,\bbB}$, $z_{i,k}^{+,\bbA}$ by $z_{i,k,\bbB}$, 
\item replace $g^i_{jk}$ by $h_i^{jk}$, $g_{ijk}$ by $h^{ijk}$ and vice versa, 
\item replace every $\bbA^i$ by $\bbB_i$, 
\item multiply by $(-1)^{\#E(\Gamma_{\bbA})}$, 
where $E(\Gamma_{\bbA})$ denotes the set of edges of $\Gamma_{\bbA}$. 
\end{itemize}
\subsection{Reducing the residual fields}\label{sec:ReducingResFields}
Naively, after pairing the states $\psi_{M_1},\psi_{M_2}$, the new state is a function on the direct sum of the spaces of residual fields $\widetilde{\mathcal{V}_M} = \mathcal{V}_{M_1} \oplus \mathcal{V}_{M_2}$. This produces a valid state in the sense that it satisfies the mQME. However, in most cases it is not the minimal possible space of residual fields and it is possible to reduce it using the methods of \cite{Cattaneo2017}. We will discuss the reduction first since it will allow us to simplify the computations later.  The reduction is given by BV pushforward along the fibers of $\calV_{M_1} \oplus \calV_{M_2} \to \calV_{M}$. These fibers form a subspace of $\widetilde{\calV}_M$ called the subspace of \emph{redshirt residual fields}. \\
In the case of lens spaces there is a significant difference between the cases $p \neq 0$ and the case $p=0$, and we will discuss these separately. 
\subsubsection{Case $p\neq 0$} 
Let $M = L_{p,q}$ for $p \neq 0$. Recall that $M_1 = M_2 = D^2 \times S^1$, $\de_1M_1 = \de_2M_2 = S^1 \times S^1 =: \T^2$ and $\de_2M_1 = \de_1M_2 = \emptyset$. Let $\varphi\colon \T^2 \to \T^2$ be the diffeomorphism  given by $t \mapsto mt + p\theta, \theta \mapsto nt + q\theta $ so that $M = M_1 \cup_{\varphi} M_2$. We have $H^{\bullet}_{Di}(M_i) = H^{\bullet}(M_i,\de M_i) \cong H^{\bullet}(D^2,S^1) \otimes H^{\bullet}(S^1)$ and $H^{\bullet}(M_i) \cong H^{\bullet}(S^1)$ for $i\neq j$.
The spaces of redshirt residual fields can be identified by  $\int_{\Sigma}\sfb_1\sfa_2 = \int_{\Sigma}\sfb_1^{\times}\sfa_2^{\times}$.  
The reduced state is then  defined by 
\begin{equation}
\check{\psi} = \int_{\calL^\times}\widetilde{\psi} \label{eq:reduction}
\end{equation} 
where $\calL^\times$ is the zero section of $T^*[-1](L_1^{\times} \oplus L_2^\times)$. In our case, we have\begin{align*}
L_1^{\times} = \langle dt \rangle, L_2^{\times} = \langle m dt + p d\theta \rangle
\end{align*}
 Taking the pushforward of the zero section amounts to contracting a pair of $z_2^{+,\bbA}$ and  $z^{2,\bbB}$ coordinates to the number 
$$ V = \Lambda^{-1} = \left(\int_{\T^2}dt\varphi^*(dt)\right)^{-1},$$  while setting $z^{2,\bbA} = z_2^{+,\bbB} = 0$. For $L_{p,q}$ we have $V = \frac{1}{p}$.  
The resulting representatives of the cohomology of $L_{p,q}$ are
\begin{align*}
\chi_1 &= 1_{S^3} \\
\chi_2 &= 
\begin{cases} \mu_1 \wedge dt_1  &\text{ on } M_1 \\ 0 &\text{ on } M_2\end{cases} 
\end{align*}
with coordinates $z^1 = z^{1,\bbB}, z^2 = z^{1,\bbA}$ and dual basis (with respect to the Poincare pairing)
\begin{align*}
\chi^1 &= 
\begin{cases}  0 &\text{ on } M_1 \\\mu_2 \wedge dt_2  &\text{ on } M_2 \end{cases}  \\
\chi^2 &= 1_{S^3}
\end{align*}
with coordinates $z^+_1 = z^{+,\bbB}_{1}, z^+_2 = z^{+,\bbA}_{1}$.  We can use this to simplify the calculation of the state on $L_{p,q}$ by ignoring pairings of diagrams that would vanish after reducing residual fields. 
\subsubsection{The case $p=0$} In this case there are no redshirt residual fields, and we do not have to perform the reduction. The resulting manifold is $M = S^2 \times S^1$. Sticking to the same conventions as above, we will get representatives of cohomology \begin{align*}
\chi_1 &= 1 \\
\chi_2 &= dt \\
\chi_3 &= \begin{cases} \mu  &\text{ on } M_1 \\ 0 &\text{ on } M_2\end{cases} \\
\chi_4 &=
\begin{cases} \mu \wedge dt  &\text{ on } M_1 \\ 0 &\text{ on } M_2\end{cases} 
\end{align*}
with coordinates $z^1 = z^{1,\bbB},z^2=z^{2,\bbB},z^3=z^{2,\bbA},z^4=z^{1,\bbA}$. The dual basis is 
\begin{align*}
\chi^1 &= 
\begin{cases}  0 &\text{ on } M_1 \\ \mu \wedge dt  &\text{ on } M_2 \end{cases}  \\
\chi^2 &= \begin{cases}  0 &\text{ on } M_1 \\ \mu   &\text{ on } M_2 \end{cases} \\ 
\chi^3 &= dt \\ 
\chi^4 &= 1 
\end{align*}
with coordinates $z_1^+=z_1^{+,\bbB}, z_2^+=z_2^{+,\bbB}, z_3^+ = z_1^{+,\bbA},z_4^+=z_2{+,\bbA}$. 
\section{The effective action on lens spaces}
\label{sec:EffectiveAction}
We now proceed to the central part of this note, the computation of the effective action on lens spaces up to two-loop order. We will consider separately the cases where $(g,V,W)$ form a Manin triple and the one where it does not. 

Since we are interested in  the two-point effective action after gluing, we have to consider all pairs of diagrams with a total of at most two interaction vertices. Also, since we are interested in the state only after reduction of the residual fields, only pairings with no residual fields or the same number of $z_2^{+,\bbA}$ and $z^{2,\bbB}$ survive, all others can be ignored. We now compute all the relevant pairings for the case where $(\g, V, W)$ form a Manin triple. 
\subsection{Case of a Manin triple}
In this case, the $\sfA^3$ and $\sfB^3$ vertices vanish and the number of diagrams to be considered is considerably reduced. 
\subsubsection{Pairing against order 0 diagram}\label{sec:order0pairing}
First we consider all pairings against the single order 0 diagram on $M_2$. Its contribution to the state, since $M_2$ is in the $\bbB$-representation, is 
$$\Psi^{\bbB}_{\Gamma_0} = -z^{k,\bbB}_1\int_{\de M}\bbB_k - z_2^{k,\bbB}\int_{\de M}dt\bbB_k.$$
We have $\varphi^*dt = m dt + p d\theta$. Hence (as used already in section \ref{sec:2pttrees}) gluing a boundary point on the $\bbA$ side against the 0-point action up to constants corresponds to multiplying the corresponding form with $1,dt$ or $d\theta$ and integrating over that boundary point. Often, it is best to perform this integration first  - the results are known from Lemma \ref{lem:ResFieldBdryIntegration} - and then compute the integral.  Notice also that this diagram does not pair to diagrams with no boundary vertices. \\
 
We can pair the two 0-point terms on either side, see figure \ref{fig:order0pairing}. The result is 
\begin{equation} \psi_{\Gamma_0}^{\bbA} * \psi_{\Gamma_0}^{\bbB} = \int_{\de M} \sfb^{\bbA}\varphi^*\sfa^{\bbB}= z_{2,k}^{+,\bbA}z^{2,\bbB} \int_{\de M}dt \varphi^*dt = p z_{2,k}^{+,\bbA}z^{2,\bbB}.\label{eq:Gamma0*Gamma0}
\end{equation}

\begin{figure}[!h]
\centering 
\begin{tikzpicture} 
\node[coordinate] (bdry) at (0,0) {};
\node[circle,draw,left, minimum size=1pt] at (-0.4,0) {$\sfb$}
edge[middlearrow={>}] (bdry);
\node[circle, draw,right] at (0.4,0) {$\sfa$}
edge[middlearrow={<}] (bdry);
\node[] (top) at (0,1) {}; 
\node[ left] at (top.west) {$M_1$};
\node[right] at (top.east) {$M_2$}; 
\node[] (bottom) at (0,-1) {}; 
\draw[thick] (bottom) -- (top); 
\tkzDrawPoints[color=black, fill=black, size=12pt](bdry);
\end{tikzpicture}
\caption{Pairing order 0 diagrams on either side.} \label{fig:order0pairing}
\end{figure}
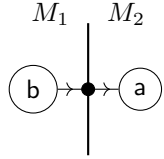

These are precisely the fields that we ought to reduce in the reduction of the redshirt residual fields as discussed in section \ref{sec:ReducingResFields}. We will perform this reduction in \ref{sec:effactionresfields}. \\
Let us turn to the 1-point diagrams. The first one ($\Gamma_{1,0}^{\bbA}$) does not contain $\bbA$ fields and hence does not pair to $\Gamma_0^{\bbB}$. The two possible pairings are shown in figure \ref{fig:order10pairing}. Using Lemma \ref{lem:ResFieldBdryIntegration} one quickly sees that the pairing in figure \ref{fig:Gamma11*Gamma0^2} vanishes. Alternatively one can use that $\int_{S^1,i}\eta_{S^1}(t_1,t_2)dt_i = 0$ for $i=1,2$.  Together with degree counting this implies the same. The remaining pairing shown in figure \ref{fig:Gamma11*Gamma0} evaluates to
\begin{equation} \psi^{\bbA}_{1,1} * \psi^{\bbB}_{\Gamma_0} = g^i_{jk}(z_{1i}^{+,\bbA}z^{1j,\bbA}-z_{2i}^{+,\bbA}z^{2j\bbA})z^{1k,\bbB} + mg^i_{jk}z_{1i}^{+,\bbA}z^{2j,\bbA}z_{1k,\bbB}.\label{eq:Gamma11*Gamma0}
\end{equation}

\begin{figure}[!h]
\centering 
\begin{subfigure}{0.3\textwidth}
\centering
\begin{tikzpicture} 
\node[coordinate] (bdry) at (0,0) {};
\node[circle, draw,right] at (0.4,0) {$\sfa$}
edge[middlearrow={<}] (bdry);
\node[coordinate] (bulk) at (-1,0) {} 
edge[-{Latex}, thick] (bdry); 
\node[circle,draw,above,minimum size=1pt] (res1) at (-1.2,0.2) {$\sfb$} 
edge[middlearrow={>}] (bulk);
\node[circle,draw,below,minimum size=1pt] (res2) at (-1.2,-0.2) {$\sfa$} 
edge[middlearrow={<}] (bulk);
\node[] (top) at (0,1.5) {}; 
\node[left] at (top.west) {$M_1$};
\node[right] at (top.east) {$M_2$}; 
\node[] (bottom) at (0,-1.5) {}; 
\draw[thick] (bottom) -- (top); 
\tkzDrawPoints[color=black, fill=black, size=12pt](bdry,bulk);
\end{tikzpicture}
\caption{$\psi_{\Gamma_{1,1}^\bbA} * \psi_{\Gamma^\bbB_0}$}\label{fig:Gamma11*Gamma0}
\end{subfigure}
\begin{subfigure}{0.3\textwidth}
\centering
\begin{tikzpicture} 
\node[coordinate] (bulk) at (-1,0) {};
\node[circle,draw,left,minimum size=1pt] (res1) at (-1.4,0) {$\sfb$} 
edge[middlearrow={>}] (bulk);
\node[coordinate] (bdry1) at (0,0.5) {}
edge[{Latex}-,thick] (bulk);
\node[circle, draw,right] at (0.4,0.5) {$\sfa$}
edge[middlearrow={<}] (bdry1);
\node[coordinate] (bdry2) at (0,-0.5) {}
edge[{Latex}-,thick] (bulk);
\node[circle, draw,right] at (0.4,-.5) {$\sfa$}
edge[middlearrow={<}] (bdry2);
\node[] (top) at (0,1.5) {}; 
\node[ left] at (top.west) {$M_1$};
\node[right] at (top.east) {$M_2$}; 
\node[] (bottom) at (0,-1.5) {}; 
\draw[thick] (bottom) -- (top); 
\tkzDrawPoints[color=black, fill=black, size=12pt](bdry1,bdry2,bulk);
\end{tikzpicture}
\caption{$\psi_{\Gamma_{1,1}^\bbA} * (\psi_{\Gamma^\bbB_0})^2$}\label{fig:Gamma11*Gamma0^2}
\end{subfigure}
\caption{Pairing 1-point diagrams on $M_1$ to 0-point diagram on $M_2$. } \label{fig:order10pairing}
\end{figure}

We can also pair against the loop diagrams, see figure \ref{fig:order20pairing} to obtain 
\begin{align}
\psi^{\bbA}_{\Gamma_{2,1}} * \psi^{\bbB}_{\Gamma_0} &= \frac{m}{12}z_{2i}^{+,\bbA}z^{2,l,\bbB}h^{ij}_kg^k_{jl} \label{eq:Gamma21*Gamma0}\\
\psi^{\bbA}_{\Gamma_{2,2}} * \psi^{\bbB}_{\Gamma_0} &= \frac{m^2}{24}z^{2,k,\bbB}z^{2,l,\bbB}h^{ij}_kg^k_{jl} \label{eq:Gamma22*Gamma0^2}
\end{align}

\begin{figure}[!h]

\begin{subfigure}[b]{0.3\textwidth}
\centering
\begin{tikzpicture} 
\node[coordinate] (bdry) at (0,1) {};
\node[circle, draw,right] at (0.4,1) {$\sfa$}
edge[middlearrow={<}] (bdry); 
\node[coordinate] (bulk1) at (-1,1) {} 
edge[-{Latex}, thick] (bdry); 
\node[coordinate] (bulk2) at (-1,-1) {};
\node[circle,draw,below,minimum size=1pt] (res1) at (-1.0,-1.3) {$\sfb$} 
edge[middlearrow={>}] (bulk2);
\draw[->,thick,-{Latex}] (bulk1) to[out=-60,in=60] (bulk2);
\draw[->,thick,-{Latex}] (bulk2) to[out=120,in=-120] (bulk1);
\node[] (top) at (0,1.5) {}; 
\node[ left] at (top.west) {$M_1$};
\node[right] at (top.east) {$M_2$}; 
\node[] (bottom) at (0,-1.5) {}; 
\draw[thick] (bottom) -- (top); 
\tkzDrawPoints[color=black, fill=black, size=12pt](bdry,bulk1,bulk2);
\end{tikzpicture}
\end{subfigure}
\begin{subfigure}[b]{0.3\textwidth}
\centering 
\begin{tikzpicture} 
\node[coordinate] (bdry1) at (0,1) {};
\node[circle, draw,right] at (0.4,1) {$\sfa$}
edge[middlearrow={<}] (bdry1);
\node[coordinate] (bdry2) at (0,-1) {};
\node[circle, draw,right] at (0.4,-1) {$\sfa$}
edge[middlearrow={<}] (bdry2);

\node[coordinate] (bulk1) at (-1,1) {} 
edge[-{Latex}, thick] (bdry1); 
\node[coordinate] (bulk2) at (-1,-1) {}
edge[-{Latex}, thick] (bdry2); 

 \draw[->,thick,-{Latex}] (bulk1) to[out=-60,in=60] (bulk2);
 \draw[->,thick,-{Latex}] (bulk2) to[out=120,in=-120] (bulk1);
\node[] (top) at (0,1.5) {}; 
\node[ left] at (top.west) {$M_1$};
\node[right] at (top.east) {$M_2$}; 
\node[] (bottom) at (0,-1.5) {}; 
\draw[thick] (bottom) -- (top); 
\tkzDrawPoints[color=black, fill=black, size=12pt](bdry1,bdry2,bulk1,bulk2);
\end{tikzpicture}
\end{subfigure}
\caption{Pairing 2-point diagrams on $M_1$ to 0-point diagram on $M_2$. } \label{fig:order20pairing}
\end{figure}
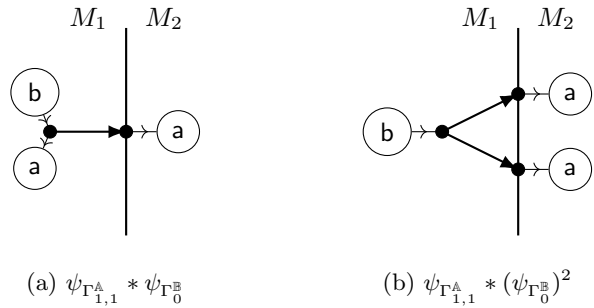

\subsubsection{Pairing the 1-point functions} 
Pairing the 1-point functions is computationally more intense, since we have to take the pullback of non-constant forms. The graph $\Gamma_{1,0}$ with no legs does not depend on boundary, hence its contribution and the one of its dual diagram simply add to the effective action. Part of it survives after reducing residual fields, and in fact we will show this is the only one-point contribution to the effective action. \\

Under the Manin triple assumption, the other pairings are the ones described in figure \ref{fig:order11pairing}. 

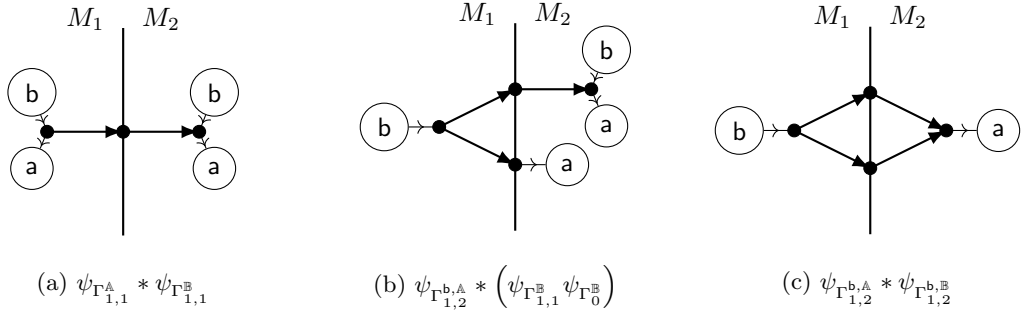
\begin{figure}[!h]
\begin{subfigure}{0.3\textwidth}
\centering
\begin{tikzpicture} 
\node[coordinate] (bdry) at (0,0) {};
\node[coordinate] (bulk1) at (-1,0) {} 
edge[-{Latex}, thick] (bdry); 
\node[circle,draw,above,minimum size=1pt] (res11) at (-1.2,0.2) {$\sfb$} 
edge[middlearrow={>}] (bulk1);
\node[circle,draw,below,minimum size=1pt] (res12) at (-1.2,-0.2) {$\sfa$} 
edge[middlearrow={<}] (bulk1);
\node[coordinate] (bulk2) at (1,0) {} 
edge[{Latex}-, thick] (bdry); 
\node[circle,draw,above,minimum size=1pt] (res21) at (1.2,0.2) {$\sfb$} 
edge[middlearrow={>}] (bulk2);
\node[circle,draw,below,minimum size=1pt] (res22) at (1.2,-0.2) {$\sfa$} 
edge[middlearrow={<}] (bulk2);
\node[] (top) at (0,1.5) {}; 
\node[ left] at (top.west) {$M_1$};
\node[right] at (top.east) {$M_2$}; 
\node[] (bottom) at (0,-1.5) {}; 
\draw[thick] (bottom) -- (top); 
\tkzDrawPoints[color=black, fill=black, size=12pt](bdry,bulk1,bulk2);
\end{tikzpicture}
\caption{$\psi_{\Gamma_{1,1}^{\bbA}} * \psi_{\Gamma_{1,1}^{\bbB}}$}
\end{subfigure}
\begin{subfigure}{0.3\textwidth}
\centering
\begin{tikzpicture}
\node[coordinate] (bdry1) at (0,0.5) {}
edge[{Latex}-,thick] (bulk1);
\node[coordinate] (bdry2) at (0,-0.5) {}
edge[{Latex}-,thick] (bulk1);
\node[circle, draw,right] at (0.4,-0.5) {$\sfa$}
edge[middlearrow={<}] (bdry2);
\node[coordinate] (bulk1) at (-1,0) {};
\node[circle,draw,left,minimum size=1pt] (res1) at (-1.4,0) {$\sfb$} 
edge[middlearrow={>}] (bulk1);
\node[coordinate] (bulk2) at (1,0.5) {} 
edge[{Latex}-, thick] (bdry1); 
\node[circle,draw,above,minimum size=1pt] (res21) at (1.2,0.7) {$\sfb$} 
edge[middlearrow={>}] (bulk2);
\node[circle,draw,below,minimum size=1pt] (res22) at (1.2,0.3) {$\sfa$} 
edge[middlearrow={<}] (bulk2);
\node[] (top) at (0,1.5) {}; 
\node[ left] at (top.west) {$M_1$};
\node[right] at (top.east) {$M_2$}; 
\node[] (bottom) at (0,-1.5) {}; 
\draw[thick] (bottom) -- (top); 
\tkzDrawPoints[color=black, fill=black, size=12pt](bdry1,bdry2,bulk1,bulk2);
\end{tikzpicture}
\caption{$\psi_{\Gamma^{\sfb,\bbA}_{1,2}} * \left(\psi_{\Gamma_{1,1}^{\bbB}}\psi_{\Gamma_0^{\bbB}}\right)$}
\end{subfigure}
\begin{subfigure}{0.3\textwidth}
\centering
\begin{tikzpicture} 
\node[coordinate] (bdry1) at (0,0.5) {}
edge[{Latex}-,thick] (bulk1);
\node[coordinate] (bdry2) at (0,-0.5) {}
edge[{Latex}-,thick] (bulk1);
\node[coordinate] (bulk1) at (-1,0) {};
\node[circle,draw,left,minimum size=1pt] (res1) at (-1.4,0) {$\sfb$} 
edge[middlearrow={>}] (bulk1);
\node[coordinate] (bulk2) at (1,0) {};
\node[circle,draw,right,minimum size=1pt] (res21) at (1.4,0) {$\sfa$} 
edge[middlearrow={<}] (bulk2);
\draw[-{Latex}, thick] (bdry1) -- (bulk2);
\draw[-{Latex}, thick] (bdry2) -- (bulk2);
\node[] (top) at (0,1.5) {}; 
\node[ left] at (top.west) {$M_1$};
\node[right] at (top.east) {$M_2$}; 
\node[] (bottom) at (0,-1.5) {}; 
\draw[thick] (bottom) -- (top); 
\tkzDrawPoints[color=black, fill=black, size=12pt](bdry1,bdry2,bulk1,bulk2);
\end{tikzpicture}
\caption{$\psi_{\Gamma^{\sfb,\bbA}_{1,2}} * \psi_{\Gamma^{\sfb,\bbB}_{1,2}}$}
\end{subfigure}
\caption{Pairing 1-point diagrams on $M_1$ to 1-point diagrams on $M_2$. } \label{fig:order11pairing}
\end{figure}
 
The next diagram $\Gamma_{1,1}$ has only constant form coefficients and the pairings are (see \cite{Wernli2018} for the computations)
\begin{align}
\psi_{\Gamma_{1,1}^{\bbA}} * \psi_{\Gamma_{1,1}^{\bbB}} &=  -n g^i_{jk}h^{kl}_m z_{1i}^{+,\bbA}z^{2j,\bbA}z^{1m,\bbB}z_{2l}^{+,\bbB}\label{eq:psi11*psi11} \\
\left(\psi_{\Gamma_{1,1}^{\bbA}}\psi_{\Gamma_0^{\bbA}}\right) * \psi_{\Gamma_{1,2}^{\sfb, \bbB}}  &=  n g^i_{jk}h^{klm}z_{1i}^{+,\bbA}z^{2j,\bbA}z_{1,l}^{+,\bbA}z_{2,m}^{+,\bbB} \label{eq:psi11psi0*psi12} \\
\psi_{\Gamma^{\sfb,\bbA}_{1,2}} * \left(\psi_{\Gamma_{1,1}^{\bbB}}\psi_{\Gamma_0^{\bbB}}\right) &=  n h_{ijk}h^{kl}_mz^{2i,\bbA}z^{1j,\bbB}z^{1m,\bbB}z_{2,l}^{+,\bbB} \label{eq:psi12*psi11psi0}\\
\psi_{\Gamma^{\sfb,\bbA}_{1,2}} * \psi_{\Gamma^{\sfb,\bbB}_{1,2}} &= g^i_{jk}h^{jk}_lz_{2i}^{+,\bbA}z^{2l,\bbB}\begin{cases} \frac{-p}{2}s(q,p) & p \neq 0 \\ \frac{q}{12} & p=0 \end{cases} \label{eq:psi12*psi12}
\end{align}
 Here $s(q,p)$ denotes the Dedekind Sum 
\begin{equation}
s(q,p) = \sum_{k=0}^{p-1}\left(\left(\frac{k}{p}\right)\right)\left(\left(\frac{qk}{p}\right)\right)\label{eq:defDedekindSum}
\end{equation}
where
$$ ((x)) = \begin{cases} 0 & x \in \Z \\
x - \lfloor x \rfloor - 1/2 & x \notin \Z \end{cases} $$ 
\subsubsection{Reducing the residual fields} \label{sec:effactionresfields}
If $p\neq 0$, we have to reduce the residual fields as discussed in \ref{sec:ReducingResFields}. We recall that this amounts to pairing $z_{2i}^{+,\bbA}$ with $z^{2,j,\bbB}$ to $\delta^j_i\cdot 1/p$ and setting their conjugates variables $z^{2,i,\bbA} = z_{2,j}^{+,\bbB} = 0.$ This eliminates many of the pairings above, namely, \eqref{eq:Gamma22*Gamma0^2},\eqref{eq:psi11*psi11},\eqref{eq:psi11psi0*psi12} and \eqref{eq:psi12*psi11psi0}. We denote the resulting effective action with $S_{eff}^{MT}$, where $MT$ stands for Manin triple. We have that 
\begin{equation}
S_{eff}^{MT} = S_{eff}^{MT,(1)} + S_{eff}^{MT,(2)} 
\end{equation}
where
\begin{equation}
S_{eff}^{MT,(1)} = g^i_{jk}\left(z^{+,\bbA}_{1i}z^{1j,\bbA}z^{1k,\bbB} +\frac{1}{2}z^{+,\bbB}_{1i}z^{1j,\bbB}z^{1k,\bbB}\right) + h^{jk}_i\left(z^{1i,\bbB}z_{1j}^{+,\bbA}z_{1k}^{+,\bbB} +\frac{1}{2}z^{+,\bbA}_{1i}z^{1j,\bbA}z^{1k,\bbA}\right)
\end{equation}
and 
\begin{equation}
S_{eff}^{MT,(2)} = g^i_{jk}h_i^{jk}\left(\frac{1}{2}s(q,p)+\frac{q+m}{12p}\right) \label{eq:SeffMT2}
\end{equation}
\subsection{The general case}
At this order, the Manin triple condition  amounts to ignoring diagrams $\Gamma^{\sfa}_{1,2}$ and $\Gamma_{1,3}$. Considering them amounts to computing the additional pairings described in figures \ref{fig:Gammaa12pairing}  and \ref{fig:Gamma13pairing}, respectively. To simplify the computations we will now drop terms that vanish after reducing fields, i.e.  we keep only terms that contain no $z^{2,i,\bbA}$- and $z_{2,j}^{+\bbB}$-variables and exactly the same number of $z_{2,i}^{+,\bbA}$ and $z^{2,j,\bbB}$ variables. In figure \ref{fig:Gammaa12pairing} this eliminates the diagrams 
\ref{fig:psi_a12_2}, \ref{fig:psi_a12_3}, \ref{fig:psi_a12_4}, together with degree counting: Ths dimension of the domain of integration is 3+2+2+3=10 in all cases, but the corresponding form degrees are 12, 11, and 14 respectively. 
On the other hand, diagram \ref{fig:psi_a12_1} yields a contribution corresponding to the $\check{\sfa}^3$ vertex on the glued lens space, while its dual will contribute the $\check{\sfb}^3$ vertex. 
\begin{figure}[!h]
\begin{subfigure}{0.3\textwidth}
\centering
\begin{tikzpicture}
\node[coordinate] (bdry1) at (0,0.5) {}
edge[{Latex}-,thick] (bulk1);
\node[circle, draw,right] at (0.4,0.5) {$\sfa$}
edge[middlearrow={<}] (bdry1);
\node[coordinate] (bdry2) at (0,-0.5) {}
edge[{Latex}-,thick] (bulk1);
\node[circle, draw,right] at (0.4,-0.5) {$\sfa$}
edge[middlearrow={<}] (bdry2);
\node[coordinate] (bulk1) at (-1,0) {};
\node[circle,draw,left,minimum size=1pt] (res1) at (-1.4,0) {$\sfa$} 
edge[middlearrow={<}] (bulk1);
\node[] (top) at (0,1.5) {}; 
\node[ left] at (top.west) {$M_1$};
\node[right] at (top.east) {$M_2$}; 
\node[] (bottom) at (0,-1.5) {}; 
\draw[thick] (bottom) -- (top); 
\tkzDrawPoints[color=black, fill=black, size=12pt](bdry1,bdry2,bulk1);
\end{tikzpicture}
\caption{$\psi_{\Gamma^{\sfa,\bbA}_{1,2}} * \left(\psi_{\Gamma_0^{\bbB}}\right)^2$}
\label{fig:psi_a12_1}
\end{subfigure}
\begin{subfigure}{0.3\textwidth}
\centering
\begin{tikzpicture}
\node[coordinate] (bdry1) at (0,0.5) {}
edge[{Latex}-,thick] (bulk1);
\node[coordinate] (bdry2) at (0,-0.5) {}
edge[{Latex}-,thick] (bulk1);
\node[circle, draw,right] at (0.4,-0.5) {$\sfa$}
edge[middlearrow={<}] (bdry2);
\node[coordinate] (bulk1) at (-1,0) {};
\node[circle,draw,left,minimum size=1pt] (res1) at (-1.4,0) {$\sfa$} 
edge[middlearrow={<}] (bulk1);
\node[coordinate] (bulk2) at (1,0.5) {} 
edge[{Latex}-, thick] (bdry1); 
\node[circle,draw,above,minimum size=1pt] (res21) at (1.2,0.7) {$\sfb$} 
edge[middlearrow={>}] (bulk2);
\node[circle,draw,below,minimum size=1pt] (res22) at (1.2,0.3) {$\sfa$} 
edge[middlearrow={<}] (bulk2);
\node[] (top) at (0,1.5) {}; 
\node[ left] at (top.west) {$M_1$};
\node[right] at (top.east) {$M_2$}; 
\node[] (bottom) at (0,-1.5) {}; 
\draw[thick] (bottom) -- (top); 
\tkzDrawPoints[color=black, fill=black, size=12pt](bdry1,bdry2,bulk1,bulk2);
\end{tikzpicture}
\caption{$\psi_{\Gamma^{\sfa,\bbA}_{1,2}} * \left(\psi_{\Gamma_{1,1}^{\bbB}}\psi_{\Gamma_0^{\bbB}}\right)$}\label{fig:psi_a12_2}
\end{subfigure}
\begin{subfigure}{0.3\textwidth}
\centering
\begin{tikzpicture} 
\node[coordinate] (bdry1) at (0,0.5) {}
edge[{Latex}-,thick] (bulk1);
\node[coordinate] (bdry2) at (0,-0.5) {}
edge[{Latex}-,thick] (bulk1);
\node[coordinate] (bulk1) at (-1,0) {};
\node[circle,draw,left,minimum size=1pt] (res1) at (-1.4,0) {$\sfa$} 
edge[middlearrow={<}] (bulk1);
\node[coordinate] (bulk2) at (1,0) {};
\node[circle,draw,right,minimum size=1pt] (res21) at (1.4,0) {$\sfa$} 
edge[middlearrow={<}] (bulk2);
\draw[-{Latex}, thick] (bdry1) -- (bulk2);
\draw[-{Latex}, thick] (bdry2) -- (bulk2);
\node[] (top) at (0,1.5) {}; 
\node[ left] at (top.west) {$M_1$};
\node[right] at (top.east) {$M_2$}; 
\node[] (bottom) at (0,-1.5) {}; 
\draw[thick] (bottom) -- (top); 
\tkzDrawPoints[color=black, fill=black, size=12pt](bdry1,bdry2,bulk1,bulk2);
\end{tikzpicture}
\caption{$\psi_{\Gamma^{\sfa,\bbA}_{1,2}} * \psi_{\Gamma^{\sfb,\bbB}_{1,2}}$}
\label{fig:psi_a12_3}
\end{subfigure}
\begin{subfigure}{0.3\textwidth}
\centering
\begin{tikzpicture} 
\node[coordinate] (bdry1) at (0,0.5) {}
edge[{Latex}-,thick] (bulk1);
\node[coordinate] (bdry2) at (0,-0.5) {}
edge[{Latex}-,thick] (bulk1);
\node[coordinate] (bulk1) at (-1,0) {};
\node[circle,draw,left,minimum size=1pt] (res1) at (-1.4,0) {$\sfa$} 
edge[middlearrow={<}] (bulk1);
\node[coordinate] (bulk2) at (1,0) {};
\node[circle,draw,right,minimum size=1pt] (res21) at (1.4,0) {$\sfb$} 
edge[middlearrow={>}] (bulk2);
\draw[-{Latex}, thick] (bdry1) -- (bulk2);
\draw[-{Latex}, thick] (bdry2) -- (bulk2);
\node[] (top) at (0,1.5) {}; 
\node[ left] at (top.west) {$M_1$};
\node[right] at (top.east) {$M_2$}; 
\node[] (bottom) at (0,-1.5) {}; 
\draw[thick] (bottom) -- (top); 
\tkzDrawPoints[color=black, fill=black, size=12pt](bdry1,bdry2,bulk1,bulk2);
\end{tikzpicture}
\caption{$\psi_{\Gamma^{\sfb,\bbA}_{1,2}} * \psi_{\Gamma^{\sfb,\bbB}_{1,2}}$}
\label{fig:psi_a12_4}
\end{subfigure}
%
\caption{Pairing diagram $\Gamma^{\sfa,\bbA}_{1,2}$ on $M_1$ to diagrams on $M_2$. Here we excluded the diagrams including $\Gamma^\bbB_{1,3}$, we will get them from symmetry from the ones for $\Gamma^\bbA_{1,3}$. 
 } \label{fig:Gammaa12pairing}
\end{figure}
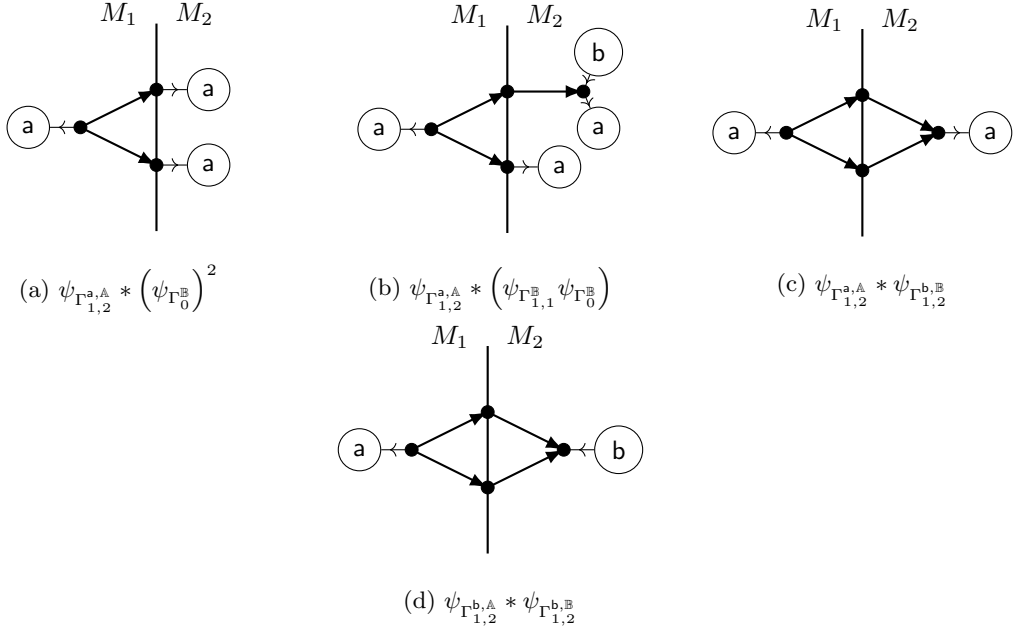

Now let us look at the diagrams in figure \ref{fig:Gamma13pairing}. The first two  diagrams in figures \ref{fig:psi_13_1},\ref{fig:psi_13_2}, and also the one in figure \ref{fig:psi_13_4} do not contribute, even if we do not reduce residual fields, this follows from the discussion on vanishing of two-point tree contributions after gluing. After reducing residual fields, diagram \ref{fig:psi_13_3} vanishes for degree reasons: only the zero-form part of $\sfa$ survives, so the total form degree is 10, while integration is over a 12-dimensional space. By degree counting, the only nonzero term contains the one-form parts of $\sfb^{\bbA}$ and $\sfa^{\bbB}$. After reducing residual fields, this corresponds to part of a theta diagram for the glued propagator, together with the last diagram \ref{fig:psi_13_6}. 
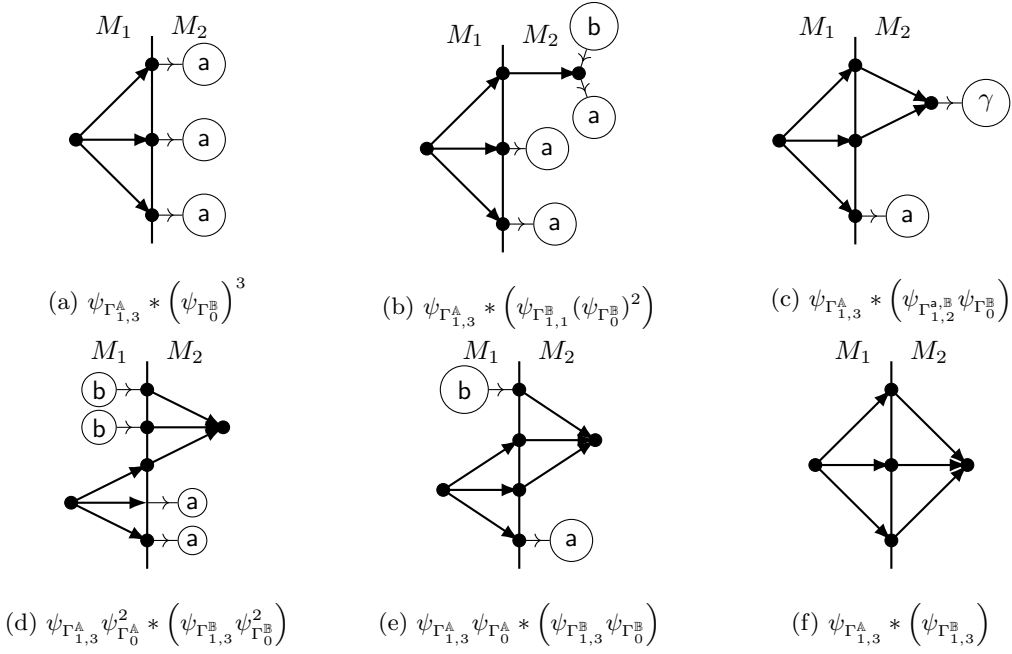
\begin{figure}[!h]
\begin{subfigure}{0.3\textwidth}
\centering
\begin{tikzpicture}
\node[coordinate] (bulk1) at (-1,0) {};

\node[coordinate] (bdry1) at (0,1) {}
edge[{Latex}-,thick] (bulk1);
\node[circle, draw,right] at (0.4,1) {$\sfa$}
edge[middlearrow={<}] (bdry1);
\node[coordinate] (bdry2) at (0,0) {}
edge[{Latex}-,thick] (bulk1);
\node[circle, draw,right] at (0.4,0) {$\sfa$}
edge[middlearrow={<}] (bdry2);
\node[coordinate] (bdry3) at (0,-1) {}
edge[{Latex}-,thick] (bulk1);
\node[circle, draw,right] at (0.4,-1) {$\sfa$}
edge[middlearrow={<}] (bdry3);
\node[] (top) at (0,1.5) {}; 
\node[ left] at (top.west) {$M_1$};
\node[right] at (top.east) {$M_2$}; 
\node[] (bottom) at (0,-1.5) {}; 
\draw[thick] (bottom) -- (top); 
\tkzDrawPoints[color=black, fill=black, size=12pt](bdry1,bdry2,bdry3,bulk1);
\end{tikzpicture}
\caption{$\psi_{\Gamma^{\bbA}_{1,3}} * \left(\psi_{\Gamma_0^{\bbB}}\right)^3$}
\label{fig:psi_13_1}
\end{subfigure}
\begin{subfigure}{0.3\textwidth}
\centering
\begin{tikzpicture}
\node[coordinate] (bulk1) at (-1,0) {};
\node[coordinate] (bdry1) at (0,1) {}
edge[{Latex}-,thick] (bulk1);
\node[coordinate] (bdry2) at (0,0) {}
edge[{Latex}-,thick] (bulk1);
\node[circle, draw,right] at (0.3,0) {$\sfa$}
edge[middlearrow={<}] (bdry2);
\node[coordinate] (bdry3) at (0,-1) {}
edge[{Latex}-,thick] (bulk1);
\node[circle, draw,right] at (0.4,-1) {$\sfa$}
edge[middlearrow={<}] (bdry3);
\node[coordinate] (bulk2) at (1,1) {} 
edge[{Latex}-, thick] (bdry1); 
\node[circle,draw,above,minimum size=1pt] (res21) at (1.2,1.3) {$\sfb$} 
edge[middlearrow={>}] (bulk2);
\node[circle,draw,below,minimum size=1pt] (res22) at (1.2,0.7) {$\sfa$} 
edge[middlearrow={<}] (bulk2);
\node[] (top) at (0,1.5) {}; 
\node[ left] at (top.west) {$M_1$};
\node[right] at (top.east) {$M_2$}; 
\node[] (bottom) at (0,-1.5) {}; 
\draw[thick] (bottom) -- (top); 
\tkzDrawPoints[color=black, fill=black, size=12pt](bdry1,bdry2,bdry3,bulk1,bulk2);
\end{tikzpicture}
\caption{$\psi_{\Gamma^{\bbA}_{1,3}} * \left(\psi_{\Gamma_{1,1}^{\bbB}}(\psi_{\Gamma_0^{\bbB}})^2\right)$}
\label{fig:psi_13_2}
\end{subfigure}
\begin{subfigure}{0.3\textwidth}
\centering
\begin{tikzpicture} 
\node[coordinate] (bulk1) at (-1,0) {};
\node[coordinate] (bdry1) at (0,1) {}
edge[{Latex}-,thick] (bulk1);
\node[coordinate] (bdry2) at (0,0) {}
edge[{Latex}-,thick] (bulk1);
\node[coordinate] (bdry3) at (0,-1) {}
edge[{Latex}-,thick] (bulk1);
\node[circle, draw,right] at (0.4,-1) {$\sfa$}
edge[middlearrow={<}] (bdry3);
\node[coordinate] (bulk2) at (1,0.5) {};
\node[circle,draw,right,minimum size=1pt] (res21) at (1.4,0.5) {$\gamma$} 
edge[middlearrow={<}] (bulk2);
\draw[-{Latex}, thick] (bdry1) -- (bulk2);
\draw[-{Latex}, thick] (bdry2) -- (bulk2);
\node[] (top) at (0,1.5) {}; 
\node[ left] at (top.west) {$M_1$};
\node[right] at (top.east) {$M_2$}; 
\node[] (bottom) at (0,-1.5) {}; 
\draw[thick] (bottom) -- (top); 
\tkzDrawPoints[color=black, fill=black, size=12pt](bdry1,bdry2,bdry3,bulk1,bulk2);
\end{tikzpicture}
\caption{$\psi_{\Gamma^{\bbA}_{1,3}} * \left(\psi_{\Gamma^{\sfa,\bbB}_{1,2}}\psi_{\Gamma_0^\bbB}\right)$}
\label{fig:psi_13_3}
\end{subfigure}
\begin{subfigure}{0.3\textwidth}
\centering
\begin{tikzpicture} 
\node[coordinate] (bulk1) at (-1,-0.5) {};

\node[coordinate] (bdry0) at (0,1) {};
\node[circle, draw,left,inner sep=1.5pt] at (-0.4,1) {$\sfb$}
edge[middlearrow={>}] (bdry0);
\node[coordinate] (bdry1) at (0,0.5) {};
\node[circle, draw,left,inner sep=1.5pt] at (-0.4,0.5) {$\sfb$}
edge[middlearrow={>}] (bdry1);
\node[coordinate] (bdry2) at (0,0) {}
edge[{Latex}-,thick] (bulk1);
\node[coordinate] (bdry4) at (0,-0.5) {}
edge[{Latex}-,thick] (bulk1);
\node[circle, draw,right,inner sep=1.5pt] at (0.4,-0.5) {$\sfa$}
edge[middlearrow={<}] (bdry4);
\node[coordinate] (bdry5) at (0,-1) {}
edge[{Latex}-,thick] (bulk1);
\node[circle, draw,right,inner sep=1.5pt] at (0.4,-1) {$\sfa$}
edge[middlearrow={<}] (bdry3);
\node[coordinate] (bulk2) at (1,0.5) {};
\draw[-{Latex}, thick] (bdry0) -- (bulk2);
\draw[-{Latex}, thick] (bdry1) -- (bulk2);
\draw[-{Latex}, thick] (bdry2) -- (bulk2);
\node[] (top) at (0,1.5) {}; 
\node[ left] at (top.west) {$M_1$};
\node[right] at (top.east) {$M_2$}; 
\node[] (bottom) at (0,-1.5) {}; 
\draw[thick] (bottom) -- (top); 
\tkzDrawPoints[color=black, fill=black, size=12pt](bdry0,bdry1,bdry2,bdry3,bulk1,bulk2);
\end{tikzpicture}
\caption{$\psi_{\Gamma^{\bbA}_{1,3}}\psi^2_{\Gamma^\bbA_0} * \left(\psi_{\Gamma^{\bbB}_{1,3}}\psi^2_{\Gamma^\bbB_0}\right)$}
\label{fig:psi_13_4}
\end{subfigure}
\begin{subfigure}{0.3\textwidth}
\centering
\begin{tikzpicture} 
\node[coordinate] (bulk1) at (-1,-0.33) {};
\node[coordinate] (bdry0) at (0,1) {};
\node[circle, draw,left] at (-0.4,1) {$\sfb$}
edge[middlearrow={>}] (bdry0);
\node[coordinate] (bdry1) at (0,0.33) {}
edge[{Latex}-,thick] (bulk1);
\node[coordinate] (bdry2) at (0,-0.33) {}
edge[{Latex}-,thick] (bulk1);
\node[coordinate] (bdry3) at (0,-1) {}
edge[{Latex}-,thick] (bulk1);
\node[circle, draw,right] at (0.4,-1) {$\sfa$}
edge[middlearrow={<}] (bdry3);
\node[coordinate] (bulk2) at (1,0.33) {};
\draw[-{Latex}, thick] (bdry0) -- (bulk2);
\draw[-{Latex}, thick] (bdry1) -- (bulk2);
\draw[-{Latex}, thick] (bdry2) -- (bulk2);
\node[] (top) at (0,1.5) {}; 
\node[ left] at (top.west) {$M_1$};
\node[right] at (top.east) {$M_2$}; 
\node[] (bottom) at (0,-1.5) {}; 
\draw[thick] (bottom) -- (top); 
\tkzDrawPoints[color=black, fill=black, size=12pt](bdry0,bdry1,bdry2,bdry3,bulk1,bulk2);
\end{tikzpicture}
\caption{$\psi_{\Gamma^{\bbA}_{1,3}}\psi_{\Gamma^{\bbA}_0} * \left(\psi_{\Gamma^{\bbB}_{1,3}}\psi_{\Gamma^\bbB_0}\right)$}
\label{fig:psi_13_5}
\end{subfigure}
\begin{subfigure}{0.3\textwidth}
\centering
\begin{tikzpicture} 

\node[coordinate] (bulk1) at (-1,0) {};
\node[coordinate] (bulk2) at (1,0) {};
\node[coordinate] (bdry1) at (0,1) {}
edge[{Latex}-,thick] (bulk1);
\node[coordinate] (bdry2) at (0,0) {}
edge[{Latex}-,thick] (bulk1);
\node[coordinate] (bdry3) at (0,-1) {}
edge[{Latex}-,thick] (bulk1);
\draw[-{Latex}, thick] (bdry1) -- (bulk2);
\draw[-{Latex}, thick] (bdry2) -- (bulk2);
\draw[-{Latex}, thick] (bdry3) -- (bulk2);
\node[] (top) at (0,1.5) {}; 
\node[ left] at (top.west) {$M_1$};
\node[right] at (top.east) {$M_2$}; 
\node[] (bottom) at (0,-1.5) {}; 
\draw[thick] (bottom) -- (top); 
\tkzDrawPoints[color=black, fill=black, size=12pt](bdry1,bdry2,bdry3,bulk1,bulk2);
\end{tikzpicture}
\caption{$\psi_{\Gamma^{\bbA}_{1,3}} * \left(\psi_{\Gamma^\bbB_{1,3}}\right)$}
\label{fig:psi_13_6}
\end{subfigure}
\caption{Pairing diagram $\Gamma^{\bbA}_{1,3}$ on $M_1$ to diagrams on $M_2$. Here $\gamma \in \{\sfa,\sfb\}$.  } \label{fig:Gamma13pairing}
\end{figure}
Below we list the weights of the glued graphs, we only list the ones yielding a non-zero contribution after reducing the residual fields. 
\begin{align}
\psi_{\Gamma^{\sfa,\bbA}_{1,2}} * \left(\psi_{\Gamma_0^{\bbB}}\right)^2 &= \frac{1}{2}g_{ijk}z^{1i\bbA}z^{1j\bbB}z^{1k\bbB} \\ 
\psi_{\Gamma^{\bbA}_{1,3}}\psi_{\Gamma^{\bbA}_0} * \left(\psi_{\Gamma^{\bbB}_{1,3}}\psi_{\Gamma^\bbB_0}\right) &=  h^{ijk}g_{ljk}z_{2i}^{+,\bbA}z^{2l,\bbB}\frac{1}{4}\int_{(\de M)^4}(\omega_{\Gamma_{0,3}})_{123}dt_4(\varphi^{\times 4})^*(\omega_{\Gamma_{0,3}})_{234}dt_1) \label{eq:psiB12psiB0psiA31ptpairing}\\
\psi_{\Gamma_{1,3}^{\sfb,\bbB}} * \psi_{\Gamma_{1,3}^{\sfb,\bbA}} &= \frac{1}{6}h^{ijk}g_{ijk}\int_{(\de M)^3}(\omega_{\Gamma_{0,3}})(\varphi^{\times 3})^*(\omega_{\Gamma_{0,3}})
\end{align}
The resulting effective action is 
\begin{equation}
S_{eff} = S_{eff}^{MT} + S_{eff}^{NMT} = S_{eff}^{MT} + S_{eff}^{NMT, (1)}  + S_{eff}^{NMT, (2)}  
\end{equation}
where 
\begin{equation}
S_{eff}^{NMT,(1)} =  \frac{1}{2}f_{ijk}z^{1i\bbA}z^{1j\bbB}z^{1k\bbB} +  \frac{1}{2}f^{ijk}z_{1i}^{+,\bbB}z_{1j}^{+,\bbA}z_{1k}^{+,\bbA}
\end{equation}
and 
\begin{equation}
S_{eff}^{NMT,(2)} = \frac{1}{2}f_{ijk}f^{ijk}\left(s(q,p) + \sum_{k=0}^{p-1}\eta_{S^1}(k/p)f(qk/p) + \eta_{S^1}(k/p)f(mk/p) + \frac{m+p}{2\pi^2}H_{1/p}\right)\label{eq:SeffNMT2}
\end{equation} where the function $f \colon S^1 \to \R$ is given by 
\begin{equation} f(\theta) = \cos(2\pi\theta)\eta_{S^1}(\theta)-\frac{1}{\pi}\sin2\pi\theta\log2|\sin\pi\theta| \label{eq:def_f}
\end{equation}
for $\theta \notin \Z$, and $f(k)= 0$ for $k \in \Z$, and $H_x$ denotes analytic extension of the harmonic numbers to $\R$. 
\subsection{Weights of oriented Theta graphs on lens spaces}
From this computation we can now extract the weights of oriented theta graphs on lens spaces by comparing with the state on the glued lens space computed with the reduced glued residual fields and the corresponding glued propagator.   We will briefly describe the first diagrams contributing to the state on any rational homology sphere, and then use the formula for the reduced propagator to identify these with the diagrams arising from the gluing described in the previous paragraphs.
\subsubsection{Low-order diagrams on rational homology spheres} 
On orientable rational homology spheres $M$ the space of residual fields is $\left(\coho(M) \oplus \coho(M)\right)[1]$. A choice of volume form $v \in H^3(M)$ gives a basis $\langle 1, v \rangle$ for $\coho(M)$. Since there is no boundary, there are no source terms. Feynman diagrams can only be closed graphs, possibly with residual fields placed the vertices. Hence the 1-point contribution to the effective action is three zero modes contracted at a single vertex and integrated over $M$. This yields a numerical coefficient which is either 0 or 1, times a cubic polynomial in the coordinate on the space of residual fields, multiplied with the structure constants of the corresponding vertex\footnote{Notice that this is true on any closed manifold where the representatives of the  residual fields are closed under wedge product.}. 
\begin{figure}[h]
\centering
\begin{tikzpicture}[scale=1]

  \node[shape=coordinate] (O) at (0,0) {};
  \node (resb) at (0,0.5) [shape=circle,above,draw,minimum size=1pt] {$\gamma$}
edge (O);
\node (resa) at (-0.3,-0.4) [shape=circle,below left,draw,minimum size=1pt] {$\gamma$}
edge (O);
\node (resa) at (0.3,-0.4) [shape=circle,below right,draw,minimum size=1pt] {$\gamma$}
edge (O);
  \tkzDrawPoint[color=black,fill=black,size=12](O)
\end{tikzpicture}
\caption{One-point function on rational homology spheres, $\gamma \in \{\sfa,\sfb\}$}
\end{figure} 
The two point contribution to the effective action consists of connected graphs with exactly two bulk vertices. If there is a single edge between them (and two residual fields on either side) the contribution vanishes since the propagator vanishes when integrated against residual fields. That leaves diagrams with two and three edges. For a diagram with two edges, we have to place a residual field at both vertices. Since residual fields are inhomogeneous forms concentrated in form degrees 0 and 3, and the propagator is a 2-form, but we integrate over the 6-dimensional space $M \times M$, these contributions also vanish. Hence we only have the oriented theta graphs \ref{fig:RHS_2ptgraphs_3} and \ref{fig:RHS_2ptgraphs_4}. Notice that the Manin triple condition rules out $\Gamma_{3,al}$. 
\begin{figure}[!h]
\centering
\begin{subfigure}{0.23\textwidth}
\centering
\begin{tikzpicture}[scale=1]
\coordinate (O) at (0,0);
 \coordinate (bulk1) at (-0.75,0) {};
 \coordinate (bulk2) at (0.75,0.0) {};
 \node[shape=circle,draw,above left] (res1) at (-0.6,0.4) {$\gamma$}
 edge (bulk1);
 \node[shape=circle,draw,above right] (res2) at (0.6,0.4) {$\gamma$}
 edge (bulk2);
 \draw[->,thick,-{Latex}] (bulk1) to[out=60,in=120] (bulk2);
 \draw[->,thick,-{Latex}] (bulk2) to[out=-120,in=-60] (bulk1);
  \tkzDrawPoints[color=black,fill=black,size=12](bulk1,bulk2);
\end{tikzpicture}
\caption{$\Gamma^{\gamma\gamma}_{2,op}$}\label{fig:RHS_2ptgraphs_1}
\end{subfigure}
\begin{subfigure}{0.23\textwidth}
\centering
\begin{tikzpicture}[scale=1]
\coordinate (O) at (0,0);
 \coordinate (bulk1) at (-0.75,0) {};
 \coordinate (bulk2) at (0.75,0.0) {};
 \node[shape=circle,draw,above left] (res1) at (-0.6,0.4) {$\gamma$}
 edge (bulk1);
 \node[shape=circle,draw,above right] (res2) at (0.6,0.4) {$\gamma$}
 edge (bulk2);
 \draw[->,thick,-{Latex}] (bulk1) to[out=60,in=120] (bulk2);
 \draw[->,thick,{Latex}-] (bulk2) to[out=-120,in=-60] (bulk1);
  \tkzDrawPoints[color=black,fill=black,size=12](bulk1,bulk2);
\end{tikzpicture}
\caption{$\Gamma^{\gamma\gamma}_{2,al}$}\label{fig:RHS_2ptgraphs_2}
\end{subfigure}
\begin{subfigure}{0.23\textwidth}
\centering
\begin{tikzpicture}[scale=1]
\coordinate (O) at (0,0);
 \coordinate (bulk1) at (-0.75,0) {};
 \coordinate (bulk2) at (0.75,0.0) {};
 \draw[->,thick,-{Latex}] (bulk1) to[out=60,in=120] (bulk2);
 \draw[->,thick,-{Latex}] (bulk1) to (bulk2);
 \draw[->,thick,-{Latex}] (bulk2) to[out=-120,in=-60] (bulk1);
  \tkzDrawPoints[color=black,fill=black,size=12](bulk1,bulk2);
\end{tikzpicture}
\caption{$\Gamma_{3,op}$}\label{fig:RHS_2ptgraphs_3}
\end{subfigure}
\begin{subfigure}{0.23\textwidth}
\centering
\begin{tikzpicture}[scale=1]
\coordinate (O) at (0,0);
 \coordinate (bulk1) at (-0.75,0) {};
 \coordinate (bulk2) at (0.75,0.0) {};
 \draw[->,thick,-{Latex}] (bulk1) to[out=60,in=120] (bulk2);
 \draw[->,thick,-{Latex}] (bulk1) to (bulk2);
 \draw[->,thick,{Latex}-] (bulk2) to[out=-120,in=-60] (bulk1);
  \tkzDrawPoints[color=black,fill=black,size=12](bulk1,bulk2);
\end{tikzpicture}
\caption{$\Gamma_{3,al}$}\label{fig:RHS_2ptgraphs_4}
\end{subfigure}
\caption{Oriented two-point diagrams. $\gamma \in \{\sfa,\sfb\}$.}\label{fig:RHS_2ptgraphs}
\end{figure}
\subsubsection{Lens spaces} 
On lens spaces one can use the decomposition described in section \ref{sec:Gluing} to compute the diagrams using the glued propagator for reduced residual fields computed in \cite{Cattaneo2017}. The result are precisely the diagrams described in sections \ref{sec:eff_action_torus}. In particular, we can identify the weights of the oriented theta diagrams with the terms in effective actions \eqref{eq:SeffMT2} and \eqref{eq:SeffNMT2}, namely  (ignoring the Lie algebra coefficients for a second)
\begin{align}
w_{\Gamma_{3,op}} &= \frac{1}{2}s(q,p) + \frac{q + m}{12p} \label{eq:ThetaManin}\\
w_{\Gamma_{3,al}} &= \frac{1}{2}s(q,p) + \sum_{k=0}^{p-1}\eta_{S^1}(k/p)f(qk/p) + \eta_{S^1}(k/p)f(mk/p) + 
\frac{q + m}{2\pi^2}H_{1/p}\label{eq:ThetaNotManin}
\end{align}

The Lie algebra coefficient of $w_{\Gamma_{3,op}}$ is given by 
\begin{equation}
e(\g) = g_{ij}^kh^{ij}_k \label{eq:defecc}
\end{equation} while the Lie algebra coefficient of $w_{\Gamma_{3,al}}$ is given by 
\begin{equation}
e'(\g) = g_{ijk}h^{ijk}. \label{eq:defecc2}
\end{equation}
We summarize the results in the following theorem that was announced in the introduction: 
\begin{thm}\label{thm:mainresult}
Consider the lens space $L_{p,q}^\varphi$ obtained from gluing $L_{p,q} = (S^1 \times D) \cup_{\varphi} (S^1 \times D) $, where $\varphi\colon S^1 \times S^1 \to S^1\times S^1$ is given by $\varphi = \begin{pmatrix}
m & p \\
n & q
\end{pmatrix}$. Then, the two-loop contribution to the state of split Chern-Simons theory on the lens space, obtained from gluing the states on two solid tori in polarizations $\mathcal{P}$ and $\calP'$  is given by 
 \begin{equation}
 w^{MT}_2 = e(\g)\left(\frac{1}{2}s(q,p) + \frac{q + m}{12p}\right)
 \end{equation}
 if  $V,W$ are subalgebras of $\g$. 
 Otherwise, the two-loop contribution is 
 \begin{equation}
 w_2^{NMT} = w_2^{MT} + e'(\g)\left(\frac{1}{2}s(q,p) + \sum_{k=0}^{p-1}\eta_{S^1}(k/p)f(qk/p) + \eta_{S^1}(k/p)f(mk/p) + 
\frac{q + m}{2\pi^2}H_{1/p}\right)
 \end{equation}
 \end{thm}
 \subsection{Example of a split Lie algebra}
 We briefly consider an exampl of a Manin triple $\mathfrak{d}$ with $e(\mathfrak{d}) \neq 0$. Recall that any Lie bialgebra $\g$ gives rise to a Manin triple $\mathfrak{d} = \g \bowtie \g^*$ via the Drinfeld double construction. We apply this for $\g = \R x \oplus \R y$ with only nontrivial bracket given by $[x,y] = y$. This is a Lie bialgebra, the Lie algebra structure on the dual is given by $[x^*,y^*] = y^*$. Hence we have $e(\mathfrak{d}) = 2$. For more examples, see \cite[Chapter 4]{Wernli2018}.
 \subsection{Dependence on choices} 
 We want to emphasize that the weights as given in Theorem \ref{thm:mainresult} depend on two choices: The choice of a gluing diffeomorphism $\varphi$, and the choice of a splitting $g = V \oplus W$. \\
 In the special case where $(\g,V,W)$ forms a Manin triple, it can be shown that $e(\g)$ does not depend on the choice of splitting $\g = V \oplus W$ as long as $V,W$ remain subalgebras, in fact it is a multiple of the quadratic Casimir of $\g$ \cite[Chapter 4]{Wernli2018}. Hence the weight of the theta graph does also not depend on it. The numeric weight of the Theta graph behaves, under changes of the gluing diffeomorphism, in an equivalent way to what was observed in the literature (see the next Section below). \\
 In the general case, however, the weight of the Theta graph computed with BV-BFV gluing does depend on the splitting (since in that case $e(\g), e'(\g)$ do). This is, of course, unexpected, since on closed spacetimes the theory does not depend on the splitting at all. Our interpretation is that there is a sort of ``polarization defect'' associated to the torus hypersurface in the glued lens space. The detailed nature of this phenomenon will be the subject of further investigation. 
 
\section{Comparing to existing results} \label{sec:Comparing}
The Theta invariant has been previously investigated by Axelrod and Singer \cite{Axelrod1991, Axelrod1994} and following up on their method by Bott and Cattaneo in \cite{Bott1998,Bott1999}. Following up on a preprint by Kontsevich (\cite{Kontsevich1994}), slightly different methods were employed by Kuperberg-Thurston \cite{Kuperberg1999} and Lescop \cite{Lescop2004a, Lescop2004, Lescop2015, Lescop2016} to investigate  the Theta invariant. Axelrod and Singer use a Riemann-Hodge gauge fixing 
$$ K_g = d_g^* \circ (\Delta_g + P_g)^{-1} $$ 
associated to a particular metric $g$ on $M$. In this formula,  $d_g^*$ is the codifferential, $\Delta_g = d d^*_g + d^*_g d$ is the Hodge-de Rham Laplacian, and $P_g$ is the projection to harmonic forms. They analyse how the weight of the Theta graph changes when the metric is varied continuously. Since in this paper we fix the axial gauge from the beginning, and this is a non-metric gauge, we cannot compare their results to ours directly. However, it is shown in the thesis \cite{Wernli2018} of the third author that the axial gauge can be approximated arbitrarily well by Riemann-Hodge gauge fixings. We aim to exploit this result in the future to connect our results to theirs. 
\subsection{Kuperberg-Thurston-Lescop Theta invariants}
On the other hand, Lescop defines a Theta invariant of \emph{framed} rational homology spheres  $(M,\tau)$ by 
\begin{equation}
\Theta(M,\tau) = \int_{C_2(M)} \omega^3 
\end{equation}
for a propagator $\omega \in \Omega^2(C_2(M,\tau))$ defined using the parallelization $\tau$. Note that there is no factor of $1/6$ like in our normalization. It is further shown \cite{Lescop2004a, Lescop2004} that 
$$ \Theta(M,g \cdot \tau) - \Theta(M,\tau) = \frac{1}{2}\deg(g) $$ where $g\colon M \to SO(3)$ and $\deg$ denotes the degree of the map, i.e. the invariant changes by a half-integer under a normalized change of framing. She concludes that 
$$ \Theta(M) := \Theta(M,\tau) - \frac{1}{4}p_1(\tau) $$ 
where $p_1(\tau)$ is a certain relative Pontryagin number associated to $\tau$, is an invariant of rational homology 3-spheres, and observes that 
$$\Theta(M) = 6\lambda(M),$$
where $\lambda(M)$ is the Casson-Walker invariant of rational homology 3-spheres, normalised to $1/2$ of Walker's original normalisation \cite{Walker1992}. 
\subsection{Comparison with weights of oriented Theta graphs}
We have shown that in the Manin triple case, the weight of the oriented Theta graph for the lens space $L(p,q)$ is (equation \eqref{eq:ThetaManin}
$$\frac{1}{2}s(q,p) + \frac{q+m}{12p}.$$ 
Note that the first term is precisely $\lambda(L_{p,q})$, since 
$\lambda_W(L_{p,q}) = s(q,p)$ (\cite{Walker1992}). On the other hand, we can change $q$ and $m$ by a multiple of $p$ by composing the gluing diffeomorphsim with a Dehn twist on one of the solid tori, thus changing the 2-framing of the resulting lens space by $
\pm 1$ (see \cite{Freed1991}). This change results in shifting the weight of the Theta graph by $\pm 1/12$.  \\ 
On the other hand, in the non-Manin triple case there is an irrational term in the weight of the Theta graph, and if we change $q$ or $m$, the weight changes by $\frac{1}{2\pi^2}H_{1/p}$ where $H_x$ is the analytic continuation of the harmonic numbers to real values. Indeed these weights have no analogue in the literature known to the authors. Still, note that up to a shift by a multiple of an irrational number they only depend on $p$ and $q (\mathrm{mod} p)$.

\section{Conclusions and Outlook}  \label{sec:Conclusion}
In this note we have computed the lowest orders of the Chern-Simons state for a split Lie algebra on lens spaces using the quantum BV-BFV formalism. We fixed most of the choices, except for the choice of gluing diffeomorphism, and polarization of the Lie algebra of coefficients. We analyzed how the state changes as we vary these choices. In particular, if we choose a polarization that is compatible with the algebraic structure we obtain results that are very similar to the ones observed in the literature on perturbative Chern-Simons invariants, namely, that the two-loop part, up to a framing dependent term, equals the Casson-Walker invariant. If we choose a polarization which is incompatible with the algebraic structure, we obtain different weights that so far have not been observed in the literature. \\
It has often been observed that the state depends heavily on the polarization, a pedagogical example can be found in \cite{Weinstein1991}. It has been proposed that in the quantization of symplectic groupoids one should always use polarizations compatible with the algebraic structure \cite{Hawkins2008}. In our case we can observe that the map 
$\calF^{\de} \to \calB$ from boundary fields to the base space of the polarization is a Lie algebra morphism if and only if we choose a Manin triple polarization.  \\
It would be interesting to see whether we can extend these results to a larger class of 3-manifolds. A first difficulty is that on higher genus handlebodies there is no axial gauge, so we cannot do explicit computations as in this paper. Another approach based on theta functions that might generalize to higher genus easier is considered in \cite{Wernli2018a}. Also, on handlebodies of nonzero Euler characteristic there might be problems in the regularization of tadpoles. 

\bibliography{BVBFV}

\begin{thebibliography}{BNGRT02}

\bibitem[AS91]{Axelrod1991}
S.~Axelrod and I.~M. Singer.
\newblock {Chern-Simons perturbation theory}.
\newblock In {\em Differential geometric methods in theoretical physics,
  Proceedings, New York}, volume~1, pages 3--45, 1991,
  \href{http://arxiv.org/abs/hep-th/9110056}{{\ttfamily arXiv:hep-th/9110056}}.

\bibitem[AS94]{Axelrod1994}
S.~Axelrod and I.~M. Singer.
\newblock Chern-{S}imons perturbation theory. {II}.
\newblock {\em J. Diff. Geom.}, 39(1):173--213, 1994,
  \href{http://arxiv.org/abs/hep-th/9304087}{{\ttfamily arXiv:hep-th/9304087}}.

\bibitem[BC98]{Bott1998}
R.~Bott and A.~S. Cattaneo.
\newblock Integral invariants of 3-manifolds.
\newblock {\em J. Diff. Geom.}, 48(1):91--133, 1998,
  \href{http://arxiv.org/abs/dg-ga/9710001}{{\ttfamily arXiv:dg-ga/9710001}}.

\bibitem[BC99]{Bott1999}
R.~Bott and A.~S. Cattaneo.
\newblock Integral invariants of 3-manifolds {II}.
\newblock {\em J. Diff. Geom.}, 53(1):1--13, 1999,
  \href{http://arxiv.org/abs/math/9802062}{{\ttfamily arXiv:math/9802062}}.

\bibitem[BF83]{Batalin1983a}
I.~Batalin and E.~Fradkin.
\newblock A generalized canonical formalism and quantization of reducible gauge
  theories.
\newblock {\em Phys. Lett. B}, 122(2):157--164, March 1983.

\bibitem[BF86]{Batalin1986}
I.~A. Batalin and E.~S. Fradkin.
\newblock Operator quantization and abelization of dynamical systems subject to
  first-class constraints.
\newblock {\em La Rivista Del Nuovo Cimento Series 3}, 9(10):1--48, oct 1986.

\bibitem[BN95]{Bar-Natan1995}
D.~Bar-Natan.
\newblock On the {V}assiliev knot invariants.
\newblock {\em Topology}, 34(2):423--472, 1995.

\bibitem[BNGRT02]{Bar-Natan2002}
D.~Bar-Natan, S.~Garoufalidis, L.~Rozansky, and D.~P. Thurston.
\newblock The Århus integral of rational homology 3-spheres {I}: A highly non
  trivial flat connection on {$S_3$}.
\newblock {\em Selecta Mathematica}, 8(3):315--339, sep 2002.

\bibitem[BRS75]{Becchi1975}
C.~Becchi, A.~Rouet, and R.~Stora.
\newblock Renormalization of the abelian {Higgs-Kibble} model.
\newblock {\em Commun. Math. Phys.}, 42(2):127--162, 1975.

\bibitem[BRS76]{Becchi1976}
C.~Becchi, A.~Rouet, and R.~Stora.
\newblock Renormalization of gauge theories.
\newblock {\em Ann. Phys.}, 98:287--321, 1976.

\bibitem[BV77]{Batalin1977}
I.~A. Batalin and G.~A. Vilkovisky.
\newblock Relativistic {S}-matrix of dynamical systems with boson and fermion
  constraints.
\newblock {\em Phys. Lett. B}, 69(3):309--312, aug 1977.

\bibitem[BV81]{Batalin1981}
I.~A. Batalin and G.~A. Vilkovisky.
\newblock Gauge algebra and quantization.
\newblock {\em Phys. Lett. B}, 102(1):27--31, June 1981.

\bibitem[BV83]{Batalin1983}
I.~A. Batalin and G.~A. Vilkovisky.
\newblock Quantization of gauge theories with linearly dependent generators.
\newblock {\em Phys. Rev. D}, 28(10):2567--2582, November 1983.

\bibitem[BW97]{Bates1997}
S.~Bates and A.~Weinstein.
\newblock {\em Lectures on the Geometry of Quantization (Berkeley Mathematical
  Lecture Notes; Vol 8)}.
\newblock American Mathematical Society, 1997.

\bibitem[CM08]{Cattaneo2008}
A.~S. Cattaneo and P.~Mnev.
\newblock Remarks on {C}hern-{S}imons {I}nvariants.
\newblock {\em Commun. Math. Phys.}, 293:803--836, November 2008,
  \href{http://arxiv.org/abs/0811.2045}{{\ttfamily arXiv:0811.2045}}.

\bibitem[CMR]{Cattaneo2016}
A.~S. Cattaneo, P.~Mnev, and N.~Reshetikhin.
\newblock Perturbative {BV} theories with {S}egal-like gluing,
  \href{http://arxiv.org/abs/1602.00741v2}{{\ttfamily arXiv:1602.00741v2}}.

\bibitem[CMR11]{Cattaneo2011}
A.~S. Cattaneo, P.~Mnev, and N.~Reshetikhin.
\newblock Classical and quantum {L}agrangian field theories with boundary.
\newblock In {\em {Proceedings, 11th Hellenic School and Workshops on
  Elementary Particle Physics and Gravity (CORFU2011)}}, volume CORFU2011,
  page~44. 2011, \href{http://arxiv.org/abs/1207.0239}{{\ttfamily
  arXiv:1207.0239}}.

\bibitem[CMR14]{Cattaneo2014}
A.~S. Cattaneo, P.~Mnev, and N.~Reshetikhin.
\newblock Classical {BV T}heories on manifolds with boundary.
\newblock {\em Commun. Math. Phys.}, 332(2):535--603, August 2014,
  \href{http://arxiv.org/abs/1201.0290}{{\ttfamily arXiv:1201.0290}}.

\bibitem[CMR15]{Cattaneo2015a}
A.~S. Cattaneo, P.~Mnev, and N.~Reshetikhin.
\newblock {Semiclassical Quantization of Classical Field Theories}.
\newblock In {\em {Mathematical Aspects of Quantum Field Theories}}, D.~Calaque
  and T.~Strobl, editors, Mathematical Physics Studies, pages 275--324.
  Springer, Cham, 2015.

\bibitem[CMR17]{Cattaneo2017}
A.~S. Cattaneo, P.~Mnev, and N.~Reshetikhin.
\newblock Perturbative quantum gauge theories on manifolds with boundary.
\newblock {\em Commun. Math. Phys.}, 357(2):631--730, dec 2017.

\bibitem[CMW17]{Cattaneo2017a}
A.~S. Cattaneo, P.~Mnev, and K.~Wernli.
\newblock Split {C}hern{\textendash}{S}imons theory in the {BV}-{BFV}
  formalism.
\newblock In {\em Quantization, Geometry and Noncommutative Structures in
  Mathematics and Physics}, pages 293--324. Springer International Publishing,
  2017.

\bibitem[CMW18]{Cattaneo2018a}
A.~S. Cattaneo, N.~Moshayedi, and K.~Wernli.
\newblock Perturbative quantization of nonlinear {AKSZ} sigma models on
  manifolds with boundary.
\newblock 2018, \href{http://arxiv.org/abs/1807.11782v1}{{\ttfamily
  arXiv:1807.11782v1}}.

\bibitem[CS74]{Chern1974}
S.-S. Chern and J.~Simons.
\newblock Characteristic forms and geometric invariants.
\newblock {\em Ann. of Math. (2)}, 99(1):48--69, 1974.

\bibitem[FG91]{Freed1991}
D.~S. Freed and R.~E. Gompf.
\newblock Computer calculation of {W}itten's $3$-manifold invariant.
\newblock {\em Commun. Math. Phys.}, 141(1):79--117, 1991.

\bibitem[Fio03]{Fiorenza2003}
D.~Fiorenza.
\newblock An introduction to the {B}atalin-{V}ilkovisky formalism.
\newblock In {\em Comptes Rendus des Rencontres Mathematiques de Glanon,
  Edition 2003}. 2003, \href{http://arxiv.org/abs/math/0402057}{{\ttfamily
  arXiv:math/0402057}}.

\bibitem[FK89]{Froehlich1989}
J.~Fr{\"o}hlich and C.~King.
\newblock The {C}hern-{S}imons theory and knot polynomials.
\newblock {\em Commun. Math. Phys.}, 126(1):167--199, 1989.

\bibitem[FP67]{Faddeev1967}
L.~D. Faddeev and V.~N. Popov.
\newblock {Feynman diagrams for the Yang-Mills field}.
\newblock {\em Phys. Lett. B}, 25(1):29--30, jul 1967.

\bibitem[Fre95]{Freed1995}
D.~S. Freed.
\newblock Classical chern-simons theory, 1.
\newblock {\em Advances in Mathematics}, 113(2):237--303, jul 1995.

\bibitem[Fre02]{Freed2002}
D.~S. Freed.
\newblock {Classical Chern-Simons Theory, Part 2}.
\newblock {\em Houston J. Math.}, 28:293--, 2002.

\bibitem[FV75]{Fradkin1975}
E.~S. Fradkin and G.~A. Vilkovisky.
\newblock Quantization of relativistic systems with constraints.
\newblock {\em Phys. Lett. B}, 55(2):224--226, feb 1975.

\bibitem[FV77]{Fradkin1977}
E.~S. Fradkin and G.~A. Vilkovisky.
\newblock {Quantization of Relativistic Systems with Constraints: Equivalence
  of Canonical and Covariant Formalisms in Quantum Theory of Gravitational
  Field}.
\newblock {\em CERN Preprint}, (CERN-TH-2332), 1977.

\bibitem[GJ87]{Glimm1987}
J.~Glimm and A.~Jaffe.
\newblock {\em Quantum Physics}.
\newblock Springer New York, 1987.

\bibitem[{Haw}08]{Hawkins2008}
E.~{Hawkins}.
\newblock A groupoid approach to quantization.
\newblock {\em J. Symplectic Geom.}, 6(1):61--125, 2008.

\bibitem[Khu04]{Khudaverdian2004}
H.~M. Khudaverdian.
\newblock {Semidensities on Odd Symplectic Supermanifolds}.
\newblock {\em Commun. Math. Phys.}, 247(2):353--390, may 2004.

\bibitem[Kon93]{Kontsevich1993}
M.~Kontsevich.
\newblock Vassiliev’s knot invariants.
\newblock {\em Adv. in Sov. Math}, 16(2):137--150, 1993.

\bibitem[Kon94]{Kontsevich1994}
M.~Kontsevich.
\newblock Feynman diagrams and low-dimensional topology.
\newblock In {\em First European Congress of Mathematics Paris, July 6–10,
  1992}, A.~Joseph, F.~Mignot, F.~Murat, B.~Prum, and R.~Rentschler, editors,
  volume 120 of {\em Progress in Mathematics}, pages 97--121. Birkhäuser
  Basel, 1994.

\bibitem[KT99]{Kuperberg1999}
G.~Kuperberg and D.~P. Thurston.
\newblock Perturbative 3-manifold invariants by cut-and-paste topology, 1999,
  \href{http://arxiv.org/abs/math/9912167}{{\ttfamily arXiv:math/9912167}}.

\bibitem[Les04a]{Lescop2004}
C.~Lescop.
\newblock On the {Kontsevich-Kuperberg-Thurston} construction of a
  configuration-space invariant for rational homology 3-spheres.
\newblock 2004, \href{http://arxiv.org/abs/math/0411088v1}{{\ttfamily
  arXiv:math/0411088v1}}.

\bibitem[Les04b]{Lescop2004a}
C.~Lescop.
\newblock Splitting formulae for the {Kontsevich-Kuperberg-Thurston} invariant
  of rational homology 3-spheres.
\newblock 2004, \href{http://arxiv.org/abs/math/0411431v1}{{\ttfamily
  arXiv:math/0411431v1}}.

\bibitem[Les15]{Lescop2015}
C.~Lescop.
\newblock A formula for the {$\Theta$}-invariant from {H}eegaard diagrams.
\newblock {\em Geom. Topol.}, 19(3):1205--1248, 2015,
  \href{http://arxiv.org/abs/1209.3219}{{\ttfamily arXiv:1209.3219}}.

\bibitem[Les16]{Lescop2016}
C.~Lescop.
\newblock A combinatorial definition of the {$\Theta$}-invariant from
  {H}eegaard diagrams.
\newblock {\em North-W. Eur. J. of Math.}, 2:17--87, February 2016,
  \href{http://arxiv.org/abs/1402.2261}{{\ttfamily arXiv:1402.2261}}.

\bibitem[LMO98]{Le1998}
T.~T. Le, J.~Murakami, and T.~Ohtsuki.
\newblock On a universal perturbative invariant of 3-manifolds.
\newblock {\em Topology}, 37(3):539--574, may 1998.

\bibitem[Mne08]{Mnev2008}
P.~Mnev.
\newblock Discrete {BF} theory.
\newblock September 2008, \href{http://arxiv.org/abs/0809.1160}{{\ttfamily
  arXiv:0809.1160}}.

\bibitem[Mne17]{Mnev2017}
P.~Mnev.
\newblock {Lectures on Batalin-Vilkovisky formalism and its applications in
  topological quantum field theory}, 2017,
  \href{http://arxiv.org/abs/1707.08096}{{\ttfamily arXiv:1707.08096}}.

\bibitem[Pol05]{Polyak2005}
M.~Polyak.
\newblock Feynman diagrams for pedestrians and mathematicians.
\newblock {\em Proc. Symp. Pure Math.}, 73:15--42, 2005,
  \href{http://arxiv.org/abs/math/0406251}{{\ttfamily arXiv:math/0406251}}.

\bibitem[Res10]{Reshetikhin2010}
N.~Reshetikhin.
\newblock Lectures on quantization of gauge systems.
\newblock In {\em New Paths Towards Quantum Gravity}, pages 125--190. Springer
  Berlin Heidelberg, 2010.

\bibitem[RT91]{Reshetikhin1991}
N.~Reshetikhin and V.~G. Turaev.
\newblock Invariants of 3-{m}anifolds via link polynomials and quantum groups.
\newblock {\em Invent. Math.}, 103(1):547–597, December 1991.

\bibitem[Sch78]{Schwarz1978}
A.~Schwarz.
\newblock The partition function of degenerate quadratic functional and
  {R}ay-{S}inger invariants.
\newblock {\em Lett. Math. Phys.}, 2(3):247--252, 1978.

\bibitem[Sch07]{Schwarz2007}
A.~Schwarz.
\newblock Topological quantum field theories.
\newblock 2007, \href{http://arxiv.org/abs/hep-th/0011260}{{\ttfamily
  arXiv:hep-th/0011260}}.

\bibitem[Sch10]{Schaetz2010}
F.~Schätz.
\newblock Invariance of the {BFV} complex.
\newblock {\em Pacific J. Math.}, 248(2):453–474, Dec 2010,
  \href{http://arxiv.org/abs/0812.2357}{{\ttfamily arXiv:0812.2357}}.

\bibitem[Sch15]{Schiavina2015}
M.~Schiavina.
\newblock {\em BV-BFV Approach to General Relativity}.
\newblock PhD thesis, Universit\"at Z\"urich, 2015.

\bibitem[Sta97]{Stasheff1997}
J.~Stasheff.
\newblock Homological reduction of constrained poisson algebras.
\newblock {\em J. Diff. Geom.}, 45(1):221--240, 1997.

\bibitem[Tyu76]{Tyutin1976}
I.~V. Tyutin.
\newblock Gauge invariance in field theory and statistical physics in operator
  formalism.
\newblock {\em Preprints of P.N. Lebedev Physical Institute, No. 39}, December
  1976, \href{http://arxiv.org/abs/0812.0580}{{\ttfamily arXiv:0812.0580}}.

\bibitem[Wal92]{Walker1992}
K.~Walker.
\newblock {\em An Extension of Casson's Invariant. (Am-126), Volume 126}.
\newblock Princeton University Press, 1992.

\bibitem[Wei87]{Weinstein1987}
A.~Weinstein.
\newblock Symplectic groupoids and poisson manifolds.
\newblock {\em Bulletin of the American Mathematical Society}, 16(1):101--105,
  jan 1987.

\bibitem[Wei91]{Weinstein1991}
A.~Weinstein.
\newblock Symplectic groupoids, geometric quantization, and irrational rotation
  algebras.
\newblock In {\em Mathematical Sciences Research Institute Publications},
  A.~Weinstein and P.~Dazord, editors, pages 281--290. Springer {US}, 1991.

\bibitem[Wer18a]{Wernli2018}
K.~Wernli.
\newblock {\em Perturbative Quantization of Split Chern-Simons Theory on
  Handlebodies and Lens Spaces by the BV-BFV formalism}.
\newblock PhD thesis, Universit\"at Z\"urich, 2018.

\bibitem[Wer18b]{Wernli2018a}
K.~Wernli.
\newblock Theta invariants and theta functions (in preparation).
\newblock 2018.

\bibitem[Wit89]{Witten1989}
E.~Witten.
\newblock Quantum field theory and the {J}ones polynomial.
\newblock {\em Commun. Math. Phys.}, 121(3):351--399, 1989.

\end{thebibliography}
\bibliographystyle{kw}
\end{document}